\documentclass[a4paper,11pt]{article}
\usepackage{amsfonts}
\usepackage{latexsym}
\usepackage{amsmath}
\usepackage{amssymb}
\usepackage{amssymb}
\usepackage{slashed}
\usepackage{upgreek}
\usepackage{mathrsfs}
\usepackage{bbm}

\usepackage{amsthm}

\theoremstyle{remark}

\newtheorem*{rmk*}{Remark}

\theoremstyle{remark}
\newtheorem*{remark}{Remark}

\usepackage[utf8]{inputenc}
\usepackage[english]{babel}

\theoremstyle{definition}
\newtheorem{definition}{Definition}[section]

\newtheorem{theorem}{Theorem}[section]
\newtheorem{corollary}{Corollary}[section]
\newtheorem{lemma}[theorem]{Lemma}

\newtheorem{prop}{{\bf Proposition}}[section]

\allowdisplaybreaks

\usepackage{relsize}
\usepackage{graphicx}
\usepackage{comment}

\renewcommand{\thefootnote}{\fnsymbol{footnote}}

\def\appendix#1{\addtocounter{section}{1}\setcounter{equation}{0}
\renewcommand{\thesection}{\Alph{section}}
\section*{Appendix \thesection\protect\indent \parbox[t]{11.15cm}{#1}}
\addcontentsline{toc}{section}{Appendix \thesection\ \ \ #1}}

\font\mybb=msbm10 at 11pt

\def\bb#1{\hbox{\mybb#1}}

\def\bZ {\bb{Z}}
\def\bR {\bb{R}}

\def\bC {\bb{C}}

\newcommand\R{{\mathrm {Ric}}}

\def\rsq {]\kern -2.0pt]}
\def\lsq {[\kern -2.0pt[}

\def\h{\widehat}
\def\be{\begin{equation}}
\def\ee{\end{equation}}

\def\rsq {]\kern -2.0pt]}
\def\lsq {[\kern -2.0pt[}

\newcommand{\bea}{\begin{eqnarray}}
\newcommand{\eea}{\end{eqnarray}}

\usepackage[usenames,dvipsnames,svgnames,table]{xcolor}	
\usepackage[normalem]{ulem}				
\usepackage{microtype}
\usepackage[colorlinks, citecolor=blue, linkcolor=blue, urlcolor=Maroon, filecolor=Maroon, linktocpage=true]{hyperref} 

\usepackage{chngcntr}
\counterwithout{equation}{section}

\begin{document}

\begin{center}
\vspace*{-1.0cm}
\begin{flushright}
\end{flushright}

\vspace{2.0cm} {\Large \bf On the rigidity of special and exceptional geometries  with torsion a closed $3$-form } \\[.2cm]

\vskip 2cm
 Georgios  Papadopoulos
\\
\vskip .6cm

\begin{small}
\textit{Department of Mathematics
\\
King's College London
\\
Strand
\\
 London WC2R 2LS, UK}\\
\texttt{george.papadopoulos@kcl.ac.uk}
\end{small}
\\*[.6cm]

\end{center}

\vskip 2.5 cm

\begin{abstract}
\noindent

Under some suitable assumptions   Riemannian manifolds $(M, g, H)$ that admit a connection $\h\nabla$ with  torsion a 3-form $H$, which is both closed $d H=0$ and $\h\nabla$-covariantly constant,   are locally isometric to a product $N\times G$, where $G$ is a semisimple group and $N$ is a Riemannian manifold with $H\vert_N=0$.  If $M$ is simply connected and complete, then by the de Rham theorem $M=N\times G$ globally. We use this to simplify the proof of similar results  for  strong  CYT and HKT manifolds that obey the above hypotheses and extend them to strong $G_2$ and $\mathrm{Spin}(7)$ manifolds with torsion.  As an application, we describe the geometry  of all complete and simply connected   $G_2$ and $\mathrm{Spin}(7)$ manifolds that satisfy the above conditions.

  Compact, strong, 8-dimensional HKT manifolds, which are not hyper-K\"ahler, admit an either $\oplus^4 \mathfrak{u}(1)$ or a $\mathfrak{u}(1)\oplus \mathfrak{su}(2)$ locally free action, otherwise, they are group manifolds.  We find  that if these Lie algebra actions can be integrated to an appropriate free action of $T^4$ or $S(U(1)\times U(2))$ Lie groups that preserves the span of three complex structures,  then these HKT manifolds are  either locally isometric  and tri-holomorphic to  $\bR\times S^3\times B^4$  or    diffeomorphic to $SU(3)$, where  $B^4= \bR\times S^3$, $\bR^4$ or $K_3$.

\end{abstract}



\newpage

\renewcommand{\thefootnote}{\arabic{footnote}}



 \setcounter{section}{0}

\numberwithin{equation}{section}

\section{Introduction}\label{intro}

 Geometries $(M, g, H)$ with metric $g$ and equipped with a  connection\footnote{Schematically, $\h\nabla=\nabla+\frac{1}{2} g^{-1} H$, where $\nabla$ is the Levi-Civita connection of the metric $g$.}, $\h\nabla$, where $H$ is  the torsion   3-form of $\h\nabla$, have found widespread applications in mathematics and physics, see \cite{Curtrightzachos}- 
 \cite{FIUV} for some selected early works.  This is  especially the case for those that the holonomy of $\h\nabla$ reduces to a proper subgroup of the orthogonal group and as consequence the underlying manifold, $M$, exhibits some additional structure. Such geometries include manifolds\footnote{Different terminology has been used to describe these manifolds with torsion in the literature. For clarity, we shall use the early terminology proposed in \cite{phgp1}, where the geometry and  twistor spaces of HKT manifolds were also investigated.}  with a K\"ahler with torsion (KT), Calabi-Yau with torsion (CYT) and  hyper-K\"ahler with torsion (HKT) structure. In such a case, the holonomy of $\h\nabla$ is contained in $U(k)\subset SO(2k)$, $SU(k)\subset U(k)\subset  SO(2k)$ and $Sp(q)\subset SO(4q)$, respectively. Of particular interest is a subclass of these  geometries that obey the {\it strong} condition -- for these the torsion 3-form $H$ is {\it closed}, $dH=0$. Following the work of Hamilton \cite{Hamilton, Hamilton0} and  Perelman \cite{Perelman, Perelman0} on Ricci flows and the generalisation by  Oliynyk,  Suneeeta, and  Woolgar \cite{OSW2, OSW} to geometries with torsion a 3-form,  strong CYT and HKT manifolds  appear as solitons in the context of generalised Ricci flow, see also \cite{Tseytlin2, Huhu} for some applications. The regularity of generalised Ricci flow has been investigated by  Streets and Tian. The same authors  also presented  applications to strong CYT manifolds in \cite{ST1, ST2}.  For a  review on generalised Ricci flow,  see \cite{GFS} and references therein.  Recently, using generalised Ricci flow techniques, it has been demonstrated that all scale invariant sigma models with compact target space are conformally invariant \cite{gpew}. This has also been extended to the heterotic sigma models in \cite{gpheterotic}.

 Strong KT and HKT structures have also appeared  on the moduli space of instantons over Hermitian and HKT manifolds established by L\"ubke and Teleman \cite{lubke}  and Moraru and Verbitsky \cite{moraru}, respectively,  see also \cite{hitchin} for analogous results on the geometry of the moduli spaces of instantons over generalised Hermitian  or equivalently strong bi-KT manifolds.  Using these results and after exploring the symmetries of the moduli space of instantons over the HKT manifold $S^1\times S^3$,   Witten \cite{witten} provided evidence  that the dual conformal theory of a certain AdS$_3$/CFT$_2$ correspondence is a sigma model with target space the moduli space of instantons on $S^1\times S^3$. The geometry and symmetries of  moduli spaces of instantons and Hermitian-Einstein connections over manifolds with KT, bi-KT,  HKT and bi-HKT structures have been also explored in \cite{gp1}, where it was also established that the moduli space of  instantons over the HKT manifold $S^1\times S^3$ is isometric to a principal fibration over a Quaternionic manifold with torsion (QKT) with fibre group that has Lie algebra $\mathfrak{u}(1)\oplus \mathfrak{su}(2)$. The relation between HKT and QKT manifolds has been known for sometime in the context of homogeneous spaces \cite{op}. The QKT structure was introduced in \cite{hopqkt} and further explored in \cite{siqkt}.

 Despite the widespread applications, surprisingly, very few examples of compact strong CYT and HKT manifolds with non-vanishing torsion, $H\not=0$, are known.  Most of the examples constructed are (locally isometric to) products $N\times G$, where $N$ is a Calabi-Yau or hyper-K\"ahler manifold and $G$ is a group manifold with a CYT or HKT structure, respectively. For strong CYT manifolds with a $\h\nabla$-covariantly constant torsion $H$, $\h\nabla H=0$, where $\h\nabla$ is the connection with torsion $H$ a 3-form, a rigidity theorem has been demonstrated in \cite{qzfz, bfg, btp} confirming that all such manifolds are products $N\times G$, where $N$ is a Calabi-Yau manifold and $G$ is a group manifold. For compact 6-dimensional strong CYT, a rigidity result has been demonstrated by Apostolov, Barbaro, Lee and Streets \cite{abls}, where under some additional assumptions such manifolds were shown to be locally isometric and holomorphic to the following spaces: a Calabi-Yau manifold; $S^3\times S^3$;  $S^3\times \bR\times \bC$.

  In the strong HKT case, there are not known compact examples that are not locally isometric to a product manifold $N\times G$. The authors of \cite{bfgv} have demonstrated that all strong HKT manifolds with $\h\nabla$-covariantly constant torsion $H$ are locally isometric and tri-holomorphic to a product $N\times G$, where $N$ is a hyper-K\"ahler manifold and $G$ is a group manifold with an HKT structure.
 In the same publication, it has also been found  that solvmanifolds do not admit a strong HKT structure.  Because of this, the question  arises on whether there exist any compact  strong HKT manifold that is not such product.

 One of the objectives of this article is to generalise the above results to Riemannian manifolds that admit a connection $\h\nabla$ with torsion a 3-form $H$ such that $H$ is both closed and $\h\nabla$-covariantly constant. This will be used to revisit the proof given for KT and HKT manifolds explained above and furthermore extend the results to manifolds that admit a  connection $\h\nabla$ whose holonomy group is contained in a suitable exceptional group.  In particular, we  demonstrate the following statement.

\begin{theorem}\label{th:gdec}
Let $(M^n, g, H)$ be a connected Riemannian  manifold, $n\geq 3$,  with metric $g$ and 3-form $H$.  If $H$ is closed and $\h\nabla$-covariantly constant with respect to the metric connection with torsion $H$, $\h\nabla H=0$, then $M$ is locally isometric to $(N, g_N, 0)\times (S, g_S, H_S)\times (R, g_R, H_R)\times (Z, g_Z, H_Z)$ and $H$ decomposes as $H=H_S\oplus H_S\oplus H_R\oplus H_Z$. The individual components in the product are as follows:
\begin{enumerate}

\item $N$ is a (possibly reducible)  Riemannian manifold with $H\vert_N=0$.

 \item $ (S, g_S, H_S)=\times_i (\Sigma_i^3, g_{{}_{\Sigma_i}}, \mathrm{dvol}(\Sigma_i))$ is a product of (orientable) 3-dimensional manifolds $\Sigma_i^3$ with $H_S$ the  sum of the volume 3-forms $\mathrm{dvol}(\Sigma_i)$ up to a constant scale factor.

 \item $(R, g_R, H_R)=\times_\alpha (R_\alpha, g_{{}_{R_\alpha}}, H_{R_\alpha})$, where  $R_\alpha$ is a reducible Riemannian manifold but irreducible as a manifold with torsion $H_{R_\alpha}$.

  \item
    $(Z, g_Z, H_Z)$ is the product of symmetric spaces $\times_\ell (Z_\ell, g_{{}_{Z_\ell}}, H_{Z_\ell})$, where $Z_\ell$ is either   a compact simply connected Lie group $K_\ell$ equipped with the bi-invariant metric and $3$-form or its dual non-compact symmetric space, $\mathrm{dim}\, K_\ell\geq 5$.
    \end{enumerate}
    Furthermore, if $M$ is also simply connected and complete, then $(M^n, g, H)=(N, g_N, 0)\times (S, g_S, H_S)\times (R, g_R, H_R)\times (Z, g_Z, H_Z)$ globally.
\end{theorem}

This statement is a collection of results obtained by Agricola, Ferreira, and Friedrich in \cite{AFF}. Here, we shall outline some of the steps in the proof of the above statement. This will be used later as
the results of the above theorem can be adapted to strong $G_2$ and $\mathrm{Spin}(7)$ manifolds, $(M^7, g, H, \varphi)$ and $(M^8, g, H, \phi)$, respectively,  with torsion a $\h\nabla$-covariantly constant 3-form $H$.   The geometries of such manifolds are described in {\bf Theorems} \ref{cor:g2} and \ref{cor:spin7}, respectively.
It is found that in all cases, such manifolds are locally isometric to either a group manifold or to a product of a hyper-K\"ahler manifold with a group manifold. We provide a list of all cases that can occur.
We also give the classification of all complete, simply connected strong $G_2$ and $\mathrm{Spin}(7)$ manifolds that satisfy the hypotheses of Theorem \ref{th:gdec}, see also the remarks that accompany the Theorems \ref{cor:g2} and \ref{cor:spin7} below.

The result in Theorem \ref{th:gdec} can be used to simplify similar statements that have already been made in the literature for strong KT, CYT and HKT manifolds,  \cite{qzfz, bfg, btp, bfgv}. The proof  for
 strong KT, strong CYT and strong HKT manifolds has the additional difficulty of demonstrating the splitting of the complex structure(s). The results   are described in the {\bf Theorems} \ref{th:KTCYT} and \ref{th:HKT} for completeness.

The outcome of all these results is that requiring $H$ to be  $\h\nabla$-covariantly constant is rather restrictive to give examples of strong KT, CYT and HKT manifolds that are not products of group manifolds with other spaces. We investigate  the question on whether the $\h\nabla$-covariantly constant requirement can be replaced with a weaker condition, like for example that the (pointwise) length, $|H|$,  of $H$ is constant. Under such an assumption and for $(M,g,H)$ a compact steady generalised Ricci soliton, an application of the formula stated in Proposition \ref{prop:gsrs} and  proven in  \cite{gpew}, implies that $H$ is harmonic. We explore this and point out that for a positive definite Riemann curvature, $H$ is $\nabla$-covariantly constant. This and some additional properties are  described in Proposition \ref{prop:Hdec}.

It has been known for sometime that compact conformally balanced strong HKT (CYT) manifolds are hyper-K\"ahler (Calabi-Yau) \cite{sigp1, sigp}. (The Lee 1-form of a conformally balanced Hermitian manifolds is exact, see e.g. \cite{gpew} for a recent explanation and definitions.) On the other hand,
compact, non-conformally balanced (i.e. non-hyper-K\"ahler), strong HKT manifolds $(M^{4q}, g, H, I_r)$,  whose  a priori (reduced) holonomy group of $\h \nabla$ connection is not a proper subgroup of $Sp(q-1)$,  admit an action of the $\oplus^4\mathfrak{u}(1)$ or the $\mathfrak{u}(1)\oplus \mathfrak{su}(2)$ Lie algebra \cite{gpheterotic}.   These Lie algebras are generated by $\h\nabla$-covariantly constant vector fields and so their action is locally free, i.e. they do not have fixed points.  If these Lie algebra actions have closed orbits, they can be integrated to a group action of  $T^4$ or $S(U(1)\times U(2))$ on $M^{4q}$. In general, the space of orbits of a such group action is an orbifold.  However, if the group action is free and preserves the span of the three complex structures $I_r$ of  $M^{4q}$,  then $(M^{4q}, g, H, I_r)$ is isometric and tri-holomorphic to principal $T^4$ or $S(U(1)\times U(2))$ bundle over an HKT or QKT manifold $B^{4(q-1)}$, respectively -- the $S(U(1)\times U(2))$ is isomorphic to $U(2)$ but it will become obvious below why considering $S(U(1)\times U(2))$ is preferential.

Recently, Brienza, Fino, Grantcharov and Verbitsky \cite{bfgv} have demonstrated that compact, simply connected, strong, 8-dimensional HKT manifolds $M^8$, which are not hyper-K\"ahler, admit the action of a $\mathfrak{u}(1)\oplus \mathfrak{su}(2)$  algebra. Thus, the holonomy requirement on the connection $\h\nabla$ imposed in \cite{gpheterotic} for the existence of such an action, which was also mentioned above,  can be removed in this case.  Moreover, if the orbits of the  $\mathfrak{u}(1)\oplus\mathfrak{su}(2)$ action are closed, then $M^8$ is  a  fibration with base space a QKT orbifold. In addition, they demonstrated that such manifolds always admit another  $\mathfrak{u}(1)\oplus \mathfrak{su}(2)$ action with closed orbits. But, the space of orbits of this new action may not admit a QKT structure.

It is also a consequence  of the results of \cite{gpheterotic} that compact, strong, 8-dimensional, HKT manifolds  either admit an action of the $\oplus^4\mathfrak{u}(1)$ or the $\mathfrak{u}(1)\oplus \mathfrak{su}(2)$ Lie algebra, otherwise, they are   group manifolds. To see the latter,  they always admit four $\h\nabla$-covariantly constant vector fields.  If the Lie algebra of these four vector fields does not close under Lie brackets, then it will generate another four $\h\nabla$-covariantly constant vector fields, which makes the 8-dimensional HKT manifold parallelizable with respect to the connection $\h\nabla$. Such manifolds are  group manifolds equipped with a bi-invariant metric and 3-form torsion.

Because of the large amount of symmetry, the possibility arises that these  8-dimensional compact HKT manifolds can be classified.
Indeed, we have demonstrated the following result:

\begin{theorem}\label{th:cla}
Let $(M^8, g, H, I_r)$  be a compact, strong, 8-dimensional,  non-hyper-K\"ahler,  HKT manifold  that admits a $\oplus^4\mathfrak{u}(1)$ or a $\mathfrak{u}(1)\oplus \mathfrak{su}(2)$ action.   If this can be integrated to a  free  action of $T^4$ or $S(U(1)\times U(2))$ Lie groups on $M^8$ that preserves the span of three complex structures $I_r$ and the space of orbits of $S(U(1)\times U(2))$ group action is simply connected,  then  $(M^8, g, H, I_r)$ is either locally isometric and tri-holomorphic to a product $\bR\times S^3 \times B^4$ for $B^4=\bR\times S^3$, $\bR^4, K_3$ or diffeomeorphic to $SU(3)$. The group manifold $SU(3)$ equipped with the bi-invariant metric and 3-form admits HKT structures.
\end{theorem}

From the assumptions of the theorem, $M^8$ is the total space of a principal bundle with fibre either $T^4$ or $S(U(1)\times U(2))$. The above theorem  can also be extended whenever the Lie algebra $\mathfrak{u}(1)\oplus \mathfrak{su}(2)$ action can be integrated  to $U(2)$ -- $S(U(1)\times U(2))=U(2)$ as group manifolds.
The proof of  Theorem \ref{th:cla} relies on four intermediate results.

The first step is the proof of the statement that the base space $B^4$ of such fibration, $M^8(B^4, S(U(1)\times U(2))$, is a {\it anti-self-dual} 4-manifold $B^4$ with {\it positive scalar curvature}, see also \cite{jggpa, jggpb}. The second result is that the {\it adjoint bundle of} $M^8$ is identified with the {\it bundle of self-dual 2-forms} on $B^4$, $\Lambda^+(B^4)$ spanned by the quaternionic structure of $B^4$. Principal bundles with fibre group $S(U(1)\times SU(2))$ over  4-manifolds are characterised by the Chern classes $(c_1, c_2)\in H^2(B^4,\bZ)\oplus H^4(B^4, \bZ)$. The third  point is that the closure of $H$ imposes the {\it topological condition}
\be
3c_1^2+2\chi+3\tau=0~,
\label{topcon}
\ee
on $B^4$ and $M^8$, where in addition  $c^2_1[B^4]\leq 0$, and $\chi$ and $\tau$ are the E\"uler and signature characteristic classes of $B^4$, respectively. Finally, the fourth result is the classification by LeBrun \cite{lebrun0}, see also \cite{lebrun},  of the topological type of compact simply connected anti-self-dual  4-manifolds that admit a metric with positive scalar curvature. They found that such manifolds  are homeomorphic to  connected sums of the complex projective space $\overline{\bC P}^2$, $\#_k \overline{\bC P}^2$ and $S^4$ -- the orientation of $B^4$ is chosen to be that of the quaternionic structure  and it is opposite from that of \cite{lebrun0}. The proof of the Theorem \ref{th:cla} follows upon comparing the topological condition arising from the closure of $H$ above with the topological data of $\#_k \overline{\bC P}^2$ and $S^4$.

This paper is organised as follows: In section 2, we consider a Riemannian manifold, $(M, g, H)$, with metric $g$ and equipped with a metric connection $\h\nabla$ with torsion the 3-form $H$ -- $H$ is restricted to be a {\it closed} 3-form. Under this assumption and after deriving a  key Bianchi identity described in Lemma \ref{le:bianchi},  we explore in Lemma \ref{le:lccc} the consequences of imposing that the torsion $H$ is $\h\nabla$-covariantly constant on the geometry of a Riemannian manifold. We find that $H$ must be {\it $\nabla$-covariantly constant}  and that $H$ obeys the {\it Jacobi identity}, where $\nabla$ is the Levi-Civita connection of $g$. Then, the proof of decomposition Theorem \ref{th:gdec} follows. The applications of this result to strong KT and CYT, and strong HKT manifolds are explained in Theorems \ref{th:KTCYT} and \ref{th:HKT}, respectively, reproducing the results of \cite{bfg, btp, bfgv}.

In section 3, we extend our results to strong $G_2$ and $\mathrm{Spin}(7)$ manifolds with torsion $H\not=0$,  which, in addition,  is  $\h\nabla$-covariantly constant.  We find that these spaces are either locally isometric to a group manifold or to a product of a group manifold with a suitable hyper-K\"ahler space.

In section 4, we evaluate the results that have been obtained so far in the literature and in this paper to argue that the condition of imposing that $H$ is $\h\nabla$-covariantly constant for CYT, HKT, and $G_2$ and $\mathrm{Spin}(7)$ with torsion manifolds is too strong to produce compact examples that are not locally isometric to products of groups and manifolds that have vanishing torsion. It also appears that imposing a weaker condition, like for example that the length of $H$  is constant over the manifold, may again be very restrictive. For example, if $H^2$ is constant, then even for Riemannian manifolds with torsion, i.e. without any additional structure, the identity in Proposition \ref{prop:gsrs} together with the Bochner-Weitzenb\"ock formula will imply that $H$ is $\nabla$-covariantly constant. This indicates that such a weaker statement may not be sufficiently weak enough to construct non-trivial examples.

In section 5, we demonstrate the proof of Theorem \ref{th:cla}. We begin with the proof of the four statements needed to establish the result. We also explore further the HKT structures on $SU(3)$. Although, $SU(3)$ admits left- and right-invariant HKT structures as a group manifold. The question remains whether these are the only ones. We point out in corollary \ref{coro:one} that if there is an additional HKT structure on $SU(3)$, it is induced by a anti-self-dual structure on $\overline{\bC P}^2$ that lies in the conformal class of the Fubini-Study metric.

In appendix A, we summarise our conventions and in appendix B, we demonstrate that a certain non-linear elliptic differential equation that arises in the investigation of 8-dimensional HKT manifolds has always a solution. The proof is based on the use of supersolutions and subsolutions and relies on an iteration procedure.

\section{Manifolds with parallel 3-form torsion}

\subsection{Riemannian manifolds with parallel 3-form torsion}

Before, we proceed to prove the result described in Theorem \ref{th:gdec}, let us state the proof of two Bianchi identities associated with the curvature $\h R$ of the connection $\h\nabla$ with torsion $H$ on the Riemannian manifold $(M^n, g, H)$. The proof can be given in global notation but it is considerably simpler to state it on an orthonormal co-frame $\{e^i: i=1, \dots, n\}$,  where  $g=\delta_{ij} e^i e^j$. See also appendix A for our conventions.

\begin{lemma}\label{le:bianchi}

Let $(M, g, H)$ be a Riemannian manifold equipped with metric $g$ and a 3-form $H$ -- $H$ is not necessarily closed, $dH\not=0$. Then, the curvature $\h R$ of the connection $\h\nabla$ with torsion $H$ satisfies the Bianchi identities

\begin{align}
3 \h R_{i[j km]}&=-\h\nabla_i H_{jkm}+ \frac{1}{2} dH_{ijkm}~,
\cr
3 \h R_{[ij k]m}&=-\frac{3}{2} \h\nabla_{[i} H_{jk]m}-\frac{1}{2} \h\nabla_m H_{ijk}-\frac{1}{2} dH_{ijkm}~.
\label{bianchi}
\end{align}

\end{lemma}

\begin{proof}
The first identity is a very well known and extensively used in the literature, see e.g. \cite{gpheterotic}. Because of this, we shall not proceed with a proof. Instead, let us focus on the proof of the second Bianchi identity.

First note that
\be
\h R_{ijkm}= R_{ijkm}+\frac{1}{2} \nabla_i H_{kjm}- \frac{1}{2} \nabla_j H_{kim}+\frac{1}{4} \delta_{pq} H^p{}_{ki} H^q{}_{jm}-\frac{1}{4} \delta_{pq} H^p{}_{kj} H^q{}_{im}~,
\ee
where $\nabla$ and $R$ is the Levi-Civita connection of $g$ and $R$ is its curvature, respectively.  Next taking the cyclic permutation of the first three indices in the above expression, we find
\begin{align}
3 \h R_{[ijk]m}&=\frac{3}{2} \nabla_{[i} H_{kj]m}-\frac{3}{2} \nabla_{[j} H_{ki]m}+\frac{3}{2} \delta_{pq} H^p{}_{[ki} H^q{}_{j]m}
\cr
&=\frac{3}{2} \nabla_{[i} H_{kj]m}-\frac{1}{2} \nabla_{m} H_{jki}+\frac{3}{2} \delta_{pq} H^p{}_{[ki} H^q{}_{j]m}-\frac{1}{2} dH_{jkim}
\cr
&
=\Big(\frac{3}{2} \h\nabla_{[i} H_{kj]m}+\frac{3}{2} \delta_{pq} H^p{}_{[ik} H^q{}_{j]m}-\frac{3}{4} \delta_{pq} H^p{}_{m[i} H^q{}_{kj]}\Big)
\cr
&
\qquad-\big(\frac{1}{2} \h\nabla_{m} H_{jki}+\frac{3}{4} \delta_{pq} H^p{}_{m[j} H^q{}_{ki]}\Big)
\cr
&\qquad+\frac{3}{2} \delta_{pq} H^p{}_{[ki} H^q{}_{j]m}-\frac{1}{2} dH_{jkim}
\cr
&
=\frac{3}{2} \h\nabla_{[i} H_{kj]m}-\frac{1}{2} \h\nabla_{m} H_{jki}-\frac{1}{2} dH_{jkim}~,
 \end{align}
 where we have used the Bianchi identity of the curvature of the Levi-Civita connection. From this, the expression of the second Bianchi identity as stated in the Lemma can be deduced after a re-arrangement of indices.
\end{proof}
\begin{remark}
 One significant point is that the right hand side  of all the above Bianchi identities  depends only on the exterior derivative  and the $\h\nabla$ covariant derivative of $H$.

The above Bianchi identities hold for the curvature $\breve R$ of the connection $\breve \nabla$ that has torsion $-H$.  The formulae can be easily deduced from those of Lemma \ref{le:bianchi} by setting $H$ to $-H$.  In particular, the first Bianchi identity reads
\be
3 \breve R_{i[j km]}=\breve\nabla_i H_{jkm}- \frac{1}{2} dH_{ijkm}~.
\ee
Notice that the covariant derivative in the right hand side is now taken with respect to the connection $\breve \nabla$.
\end{remark}

\begin{lemma}\label{le:lccc}

Let $(M, g, H)$ be a Riemannian manifold as in the previous lemma. If $H$ is closed, $d H=0$,  and $\h\nabla$-covariantly constant, $\h\nabla H=0$, then $H$ is covariantly constant with respect to the Levi-Civita connection, $\nabla H=0$, and $H$ satisfies the Jacobi identity.

\end{lemma}

\begin{proof}

It is known that if $H$ is closed, then
\be
\h R_{ijkm}=\breve R_{kmij}~.
\label{bianchi3}
\ee
As $H$ is closed and $\h\nabla$-covariantly constant, a consequence of the Bianchi identities of Lemma \ref{le:bianchi} is that
\be
\h R_{[ijk]m}=\h R_{i[jkm]}=0~.
\ee
However, using the remark above and the closure of $H$
\be
\breve\nabla_i H_{jkm}=3 \breve R_{i[j km]}=-3 \h R_{[km j]i}=0~,
\ee
where we have used the identity (\ref{bianchi3}).  Therefore, $H$ is also $\breve \nabla$-covariantly constant. As $H$ is covariantly constant with respect to both $\h\nabla$ and $\breve\nabla$, it is clearly covariantly constant with respect to the Levi-Civita connection $\nabla$ and satisfies the Jacobi identity
\be
H^p{}_{ij} H_{pkm}+\mathrm{cyclic}~ (j,k,m)=0~.
\label{jac}
\ee
This completes the proof of the Lemma.
\end{proof}

\begin{remark}
Before, we proceed with the proof of Theorem \ref{th:gdec}, notice that all oriented 3-dimensional manifolds $(\Sigma^3, g_{{}_{\Sigma^3}}, \mathrm{dvol}(\Sigma^3))$, where $\mathrm{dvol}(\Sigma^3)$ is the volume 3-form, satisfy all the requirements of the theorem.  The same applies for products of such manifolds.  Such manifolds  contribute to the $S$ subspace in the decomposition of $M$ in the Theorem \ref{th:gdec}.

Furthermore, if $(M^4, g, H)$ is an oriented 4-dimensional manifold that satisfies the requirements of the Theorem \ref{th:gdec}, then it is locally isometric to $\bR\times \Sigma^3$, with $H$ identified with  $\mathrm{dvol}(\Sigma^3)$ up to rescaling with a constant. To see this, observe that the Hodge dual  $1$-form $Y$ of $H$, $Y={}^*H$,  is $\nabla$-covariantly constant. So $M^4$ locally isometrically decomposes to $\bR\times \Sigma^3$. Moreover in four dimensions, $\iota_{Y^\flat} H=0$.  As $dH=0$, one also has $\mathcal{L}_{Y^\flat} H=0$. Thus, $H$ is a $3$-form on $\Sigma^3$ and so it is a constant multiple of the volume form.
\end{remark}

\begin{remark}
The decomposition of $(M^n, g, H)$ can contain reducible Riemannian manifolds that are irreducible as manifolds with torsion because $H$ does not accordingly decompose. Examples of such manifolds are simple Lie algebras with a bi-invariant metric, which are clearly decomposable as Riemannian manifolds because they are vector spaces equipped with the flat metric, but the structure constants that are identified with $H$ are incompatible with such a decomposition.  Such spaces contribute in the $R$ subspace of the decomposition of $(M^n, g, H)$ described in Theorem \ref{th:gdec}.
\end{remark}

Suppose that the torsion $3$-form $H$ of  $(M^n, g, H)$ is closed and $\h\nabla$-covariantly constant. Choose a point $p\in M^n$ and define the Lie algebra
\be
\mathfrak{g}_T\equiv \{\iota_V H\vert V\in T_pM^n\}\subseteq \mathfrak{so}(n)=\Lambda^2(T_pM)~.
\ee
The vector space $\mathfrak{g}_T$ is a Lie algebra because $H$ satisfies the Jacobi identity at every point $p\in M^n$. In fact $(T_pM^n, \mathfrak{g}_T, H_p)$ define a skew holonomy system \cite{simons, olmos1, olmos2}.

\begin{lemma}
Suppose that the $3$-form $H$ of $(M^n, g, H)$ is closed and $\h\nabla$-covariantly constant. Let $\oplus^q_{a=1}\mathcal{T}_a$ be the decomposition of tangent space  $TM^n$ of $M^n$ in parallel $\h\nabla$-distributions. Then, all  $\mathcal{T}_a$ are integrable $\nabla$-parallel distributions and $(M^n, g, H)$ locally decomposes as $\times_{a=1}^q (W_a, g_a, H_a)$.  Moreover, if $M^n$ is complete and simply connected, then $(M^n, g, H)=\times_{a=1}^q (W_a, g_a, H_a)$ globally.
\end{lemma}
\begin{proof}
A brief outline of a proof, which can be found in \cite{AFF},  is as follows: As $H$ satisfies the Jacobi identity, it decomposes as $H=\oplus_{a=1}^q H_a$ with $H_a$ a section in  $\Lambda^3(\mathcal {T}_a)$. Moreover, as $H$  is $\nabla$-covariantly constant, this decomposition commutes with the Riemann curvature $R$ of $M^n$ and therefore as a consequence of the de Rham theorem locally $(M^n, g)$ is isometric to $\times_{a=1}^q (W_a, g_a)$ with $\mathcal{T}_a=TW_a$.  Furthermore, using the closure of $H$, one can show that $H_a$ has support on $W_a$. Thus, $(M, g, H)$ locally isometric to
 \be
 \times_{a=1}^q (W_a, g_a, H_a)~.
 \label{dcomp}
 \ee
 If $M^n$ is complete and simply connected, then $(M^n, g, H)$ decomposes as above globally.
 \end{proof}

 \begin{remark}
 Suppose that $(M^n,g,H)$ satisfies the assumptions of the Lemma above. Define the quadratic form
\be
h(X, Y)=(\iota_X H, \iota_Y H)~,~~~  h_{ij}=\frac{1}{2} H_{ipq} H_j{}^{pq}~.
\label{hHH}
\ee
As $H$ is $\nabla$-covariantly constant, $h$ is $\nabla$-covariantly constant as well. A consequence of the integrability condition of $\nabla h=0$ is that the Riemann tensor $R$ decomposes  according to the eigenspace decomposition of $h$. Such a decomposition of the tangent space $TM^n$ of $M^n$ exists as $h$ is symmetric and so it can always be diagonalised at each point $p\in M^n$. In particular, this leads to the decomposition of the tangent
space $TM^n$ of $(M^n, g, H)$ as $\mathcal{T}_0\oplus \mathcal{T}_0^\perp$, where $\mathcal{T}_0$ is the subspace associated with the zero eigenvalue of $h$ and $\mathcal{T}_0^\perp$ is its orthogonal complement. These are $\nabla$-parallel integrable distributions and $(M, g, H)$ locally decomposes as
$(N_0, g_0, 0)\times (N_1, g_1, H)$ -- it is straight forward to demonstrate that $H$ has support in $N_1$. The distribution $\mathcal{T}_0$ coincides with that of the kernel of $H$, i.e. $\mathrm{ker}\, H=\{X\in TM^n\vert \iota_X H=0\}$.  Indeed,   all the components of $H$ along the tangent space of $N$ vanish as
\be
h(V, V)=(\iota_V H, \iota_V H)=0\Longrightarrow \iota_V H=0~,
\ee
for any eigenvector $V$ of $H$ with zero eigenvalue. Then, $\mathcal{L}_V H=\iota_V dH+d\iota_V H=0$ as $H$ is closed, $dH=0$.  Clearly, the domain of dependence of $H$ is in $N_1$.
\end{remark}

 One of the key results that clarifies the decomposition of manifolds with torsion a $3$-form is described in the theorem below.
 \begin{theorem}\label{th:simple}
 Let $(M^n, g, H)$ be an irreducible, complete and simply connected Riemannian manifold with $n\geq 5$ and $H\not=0$ closed and $\h\nabla$-covariantly constant.
 If $\mathfrak{g}_T$ acts irreducibly on the tangent space of $M^n$, then $M^n$ is a compact simple Lie group $K$ with the bi-invariant metric or its dual non-compact symmetric space.
 \end{theorem}
\begin{proof}
The proof is given in \cite{AFF} based on the work of Olmos and Reggiani on the  skew holonomy systems \cite{olmos1, olmos2} and it will not be repeated here. The skew holonomy system is given by the triple $(G_T, \mathcal{V}, H)$, where $G_T$ is the Lie group with Lie algebra $\mathfrak{g}_T$ and $\mathcal{V}=T_pM^n$.  However, it is worth stressing that one of the assumptions is that $M^n$ must be irreducible as a Riemannian manifold, otherwise the theorem does not hold. Note also that although $K$ is equipped with the bi-invariant metric and the torsion $3$-form is also bi-invariant, the connection $\h\nabla$ is not necessarily the parallelisable connection of $K$; $\h\nabla$ is one of the $1$ parameter family of Cartan-Schouten connections $\nabla+\lambda H$, $\lambda\in \bR$. Another point worth mentioning is that for $n=3$, which is excluded by the assumptions of the theorem, the analysis of the skew holonomy systems allows for the possibility that for any Riemannian $3$-dimensional manifold to occur whose holonomy is $SO(3)$.
\end{proof}

Next let us focus on the  ${\mathbf{proof\, of\, Theorem}}$ \ref{th:gdec}.
\begin{proof}
We have to demonstrate that $(M^n, g, H)$ decomposes as
\be
(N, g_N, 0)\times (S, g_S, H_S)\times (R, g_R, H_R)\times (Z, g_Z, H_Z)~.
\ee
It is clear that each component $(W_a, g_a, H_a)$ with $H_a\not=0$ in the decomposition (\ref{dcomp}) is irreducible, i.e. it cannot be decomposed further as a geometry with torsion. If, in addition, $W_a$ is irreducible as a Riemannian manifold and $\mathrm{dim}\, W_a\geq 5$, then $W_a$ must be a compact simple group or its dual non-compact symmetric space as a consequence of Theorem \ref{th:simple}.  The collection of all these spaces
is denoted with the $Z$ subspace of $M^n$ in the decomposition of theorem \ref{th:gdec}.

Clearly, the  subspaces  $(W_a, g_a, H_a)$ with $H_a=0$ in the product (\ref{dcomp}) are collected in the $N$ subspace of the decomposition of $M^n$.
In a similar way, the collection of all $(W_a, g_a, H_a)$ subspaces with $\mathrm{dim}\, W_a= 3$ is denoted with the $S$ subspace of $M^n$ in the decomposition of the theorem \ref{th:gdec}.  We have also seen that if $\mathrm{dim}\, W_a= 4$, then $W_\alpha$ decomposes into a $3$-dimensional subspace that belongs to $S$ and another one that is contained in $N$.

Finally the subspace $R$ in the decomposition of $M^n$ of Theorem \ref{th:gdec} is the collection of all $(W_a, g_a, H_a)$ with $H_\alpha\not=0$ and $\mathrm{dim}\, W_\alpha\geq 5$ subspaces in the decomposition of $(M^n, g, H)$ that, although irreducible as manifolds with torsion, they are reducible as Riemannian manifolds. This completes the proof.
\end{proof}

\begin{remark}
The decomposition (\ref{dcomp}) is not natural, especially when it includes $3$-dimensional components. As any $3$-dimensional manifold can occur, it can include $3$-dimensional groups, like for example $SU(2)$ with the bi-invariant metric and $3$-form. Then, such components can be interpreted as either contained in $G$ or in $S$ in the decomposition. From now on, we shall always assume that $G$ is the maximal subspace of the decomposition (\ref{dcomp}) that is a manifold with torsion a $3$-form compatible with a group structure.
\end{remark}

\begin{corollary}
Assuming that $M$ satisfies the assumptions of Theorem \ref{th:gdec} and $(M, g, H)=(W_1, g_1, H_1)\times (W_2, g_2, H_2)$,  the connection $\h \nabla$ of $M$ decomposes as $\h\nabla=\nabla^1\oplus\h\nabla^2$. Therefore, if $X$ is a $\h\nabla$-covariant constant vector field on $M$ and $X=X_1+X2$, where $X_1$ and $X_2$ are the components of $X$ in $T_pW_1$ and $T_p W_2$ for every $p\in M$, respectively, then
\be
\h\nabla X=0\Longrightarrow \h\nabla^1 X_1=0~\mathrm{and}~ \h\nabla^2 X_2=0~,
\ee
and the domain of $X_1$ and $X_2$ is on $W_1$ and $W_2$, respectively.  Moreover if $(W_1, g_1, H_1)=(W_1, g_1, 0)$, then the above condition gives $\nabla^1 X_1=0$, i.e. $X_1$ is covariantly constant with respect to the Levi-Civita connection of $W_1$.
This condition is extended to all $\h\nabla$-covariantly constant tensor fields on $M$.
\end{corollary}
\begin{proof}
This  follows from the Theorem \ref{th:gdec} and in particular the decomposition of the metric and 3-form $H$ which is compatible with the product $W_1\times W_2$.
\end{proof}

\subsection{KT and CYT manifolds with parallel 3-form torsion}

The main result of Theorem \ref{th:gdec} is that $(M^n, g, H)$ with $H$ both closed and $\h\nabla$-covariantly constant locally decomposes as
\be
(N, g_N, 0)\times (S, g_S, H_S)\times (R, g_R, H_R)\times (Z, g_Z, H_Z)~.
\label{thdcomp}
\ee
Moreover, if $(M^n, g, H)$ is complete and simply connected, then $(M^n, g, H)$ decomposes as above globally.

The  task in this section is to adapt the results of  Theorem \ref{th:gdec} to strong KT and CYT manifolds.  For a recent concise review of the properties of KT, CYT and HKT geometries, see \cite{gpew}. Decomposition results for such manifolds with $\h\nabla$-covariantly constant $3$-form $H$ have been obtained before in \cite{qzfz, bfg, btp}.   In particular, for strong KT manifolds $(M^{2k}, g, I, H)$ with complex structure $I$ and $\h\nabla$-covariantly constant $3$-form $H$, the Theorem A of \cite{btp} states that they locally isometrically  decompose as
\be
(N, g_N, 0)\times (\bR^k, g_{\bR^k}, 0)\times (S, g_S, H_S)\times (K, g_K, H_K)~,
\ee
where $N$ is a K\"ahler manifold with metric $g_N$, $S$ is a product of  $3$-dimensional Sasakian manifolds and $K$ is a semisimple group -- the subspace $(\bR^k, g_{\bR^k}, 0)$, with $g_{\bR^k}$ the standard flat metric, has been inserted in the decomposition such that the product of all last three subspaces admits a complex structure. Thus from the perspective of Theorem \ref{th:gdec}, strong KT manifolds are of NSZ-type, i.e. the $R$ subspace in the decomposition (\ref{thdcomp}) either does not occur or it can be rewritten in terms of the other subspaces.

 Here, we shall describe an adaptation  of Theorem \ref{th:gdec} for strong CYT manifolds $(M^{2k}, g, I, H)$.  For this, we shall use the following lemma.

 \begin{lemma}\label{le:group}
 Let $(M^{2k}, g, I, H)$ be a strong  CYT manifold  with $H\not=0$ that satisfies the assumptions of Theorem \ref{th:gdec}. Then, in the decomposition
 (\ref{thdcomp}), the universal cover of $S$ is a product of $3$-spheres equipped with the round metric and $Z$ is a compact semi-simple group equipped with one of parallelising connections.
 \end{lemma}
 \begin{proof}
 All strong CYT manifolds are steady generalised Ricci solitons. In particular, they satisfy the equation
 \be
 \h\R=\h\nabla\theta
 \ee
 where $\theta$ is the Lee form of the complex structure $I$, see appendix A. As $\theta$ is expressed in terms of $H$ and $I$, $\theta$ is $\h\nabla$-covariantly constant and so
 $\h\R=0$.  This implies that the Ricci curvature of such a CYT manifold can be expressed as
 \be
 \R(X,Y)=\frac{1}{2} h(X,Y)~.
 \label{rh}
 \ee
 where $h$ is given in (\ref{hHH}).

 The above relation (\ref{rh}) between the Ricci curvature and $h$ restricts to every irreducible subspace, as a manifold with torsion,  in the decomposition (\ref{thdcomp}).  In particular, restricting this to one of the $(\Sigma^3_i, g_{{}_{\Sigma^3_i}}, \mathrm{dvol}(\Sigma^3_i))$ subspaces of $S$, it will imply that it is an Einstein space with positive scalar curvature and so its universal cover is a $3$-sphere with the standard round metric. In fact, it is useful to consider it as the $SU(2)$ group manifold with the bi-invariant metric.  Thus, the universal cover of  $S$ may be considered as a product of $SU(2)$ group manifolds equipped with a bi-invariant metric.

 Similarly, applying the relation (\ref{rh}) to each of the subspaces $Z_\ell$ of  the $Z$ subspace of the decomposition (\ref{thdcomp}), one concludes that each subspace must be a compact simple Lie group $K_\ell$ -- the possibility that $Z_\ell$ is the dual non-compact space of some $K_\ell$ is excluded because (\ref{rh}) implies that the Ricci curvature must be positive while the dual non-compact spaces have negative curvature.  Furthermore, the $\h\nabla\vert_{K_{\ell}}$ connection must be one of the parallelising connections of $K_\ell$. This is again required for (\ref{rh}) to hold.  In fact $H_{K_\ell}$ is determined up to a sign in terms of the torsion $3$-form of the parallelising connection. The choice of sign corresponds  to parallelising $K_\ell$ with respect to the left or the right group action.
 \end{proof}

 \begin{remark}
 The above Lemma applies to all compact steady generalised Ricci solitons, see definition \ref{def:rsol},  with $\h\nabla$-covariantly constant $3$-form $H$. This is because for all such manifolds
 $\h\R=0$, see Proposition \ref{prop:gsrs}. The rest of the argument follows.
 \end{remark}

\begin{theorem}\label{th:KTCYT}
 Let $(M^{2k}, g, I, H)$ be a strong  CYT manifold  with $H\not=0$ that satisfies the assumptions of Theorem \ref{th:gdec} and  with a decomposition (\ref{thdcomp})   of NSZ-type, then $M^{2n}$  isometrically  and holomorphically locally decomposes into a product $N\times K$, where  $N$ is a Calabi-Yau manifold and $K$ is a KT group manifold, not necessarily semisimple,  equipped with the parallelising  connection with torsion. Moreover, if $M$ is simply connected and complete, then $M=N\times K$ globally.
\end{theorem}
\begin{proof}
We have already demonstrated in Lemma \ref{le:group} that $M^{2k}$ is locally isometric as a Riemannian manifold to $N_0\times G$, where $N_0$ is a Riemannian manifold whose tangent vectors are in the kernel of the quadratic form $h$ -- the zero eigenvalue subspace -- and $G=S\times Z$ is a compact simply connected semisimple Lie group equipped with a parallelising connection with torsion.  At each point $p\in M^{2k}$, the tangent space splits as
\be
T_p M=T_p N_0\oplus T_p G~,
\ee
where the splitting is orthogonal with respect to the metric $g$.

There are two cases to consider. One is that $I (T_p G)\subseteq T_p G$, i.e the complex structure $I$ preserves the subspace $T_p G$. If this is the case, then
$I (T_p N_0)\subseteq T_p N_0$. This follows from the Hermiticity of the KT metric. Indeed if $I(X)=Y+Z$ with $Y\in T_pG$ and $X, Z \in T_pN_0$, then
 \be
 0=-g(I(Y), X)=g(Y, I(X))=g(Y, Y+Z)=g(Y,Y)~,
 \ee
 as $I(Y)\in T_pG$.  Thus $Y=0$ and $I(T_pN_0)\subseteq T_pN_0$, i.e.  $I$ decomposes as $I=I_G\oplus I_{N_0}$.  We have already demonstrated that the $\h\nabla$ covariant derivative decomposes as $\h\nabla= \nabla^{N_0}\oplus \h\nabla^G$, where $\nabla^{N_0}$ is the Levi-Civita connection on $N_0$ and $\h\nabla^G$ is the connection with torsion $H$ on $G$. Then, $\h\nabla I=0$ implies that the domain of $I_G$ and $I_{N_0}$ is $G$ and $N_0$, respectively. $I_{N_0}$ is integrable on $N_0$ as it is $\nabla^{N_0}$-covariantly constant with $\R=0$, i.e. $N_0$ is a Calabi-Yau manifold. Thus $M^{2k}$ locally decomposes into a product $N_0\times G$, where $N_0$ is a Calabi-Yau manifold and $G$ is a group with a KT structure establishing the theorem for KT manifolds in this case after setting $N=N_0$.

However, another possibility is that $I T_p G\nsubseteq T_p G$. To investigate this, identify the Lie algebra of $G$, $\mathfrak{g}$, with the left-invariant vector fields $\{ L_a; a=1, \dots, \mathrm{dim}\, G\}$ on $G$ and $\h\nabla^G L_a=0$, i.e. these are $\h\nabla$-covariantly constant. Then consider the span
\be
\mathfrak{k}=\mathrm{Span}\,\langle L_a, I(L_a)\rangle_{ L_a\in\mathfrak{g}}~.
\ee
 If all $I(L_a)\in \mathfrak{g}$, then $I T_p G\subseteq T_p G$ and this case has already been investigated above. Thus, there must exist some vector fields $I(L)$, where $L$ is a left-invariant vector field on $G$, that are not contained in $\mathfrak{g}$, i.e $I(L)=I(L)_{N_0}+I(L)_{G}$. To continue, notice that all $I(L_a)$ are $\h\nabla$-covariantly constant as both $I$ and $L_a$ are $\h\nabla$-covariantly constant. As a result, their Lie algebra bracket can be expressed in terms of $H$, which is a property of symmetries of a KT structure.  Thus $\mathfrak{k}$ can be identified with the Lie algebra of a group $K$ whose structure constants are given by $H$. In  this case, $K$ is not semisimple.  Moreover, $\h\nabla I(L)=0$ implies that $\h\nabla I(L)_{N_0}=0$ and $\nabla I(L)_{G}=0$ for $L$ in $\mathfrak{g}$.

Clearly, $\mathfrak{k}$ can be orthogonally decomposed as
\be
\mathfrak{k}=\mathfrak{g}\oplus\mathfrak{m}~,
\ee
where the vector fields spanning $\mathfrak{m}$ are $\nabla$-covariantly constant and therefore they commute. Moreover, since $[\mathfrak{m}, \mathfrak{g}]=0$, as the components of $H$ along $\mathfrak{m}$ vanish, the decomposition above is compatible with the Lie algebra structure.  The integrability condition of the $\nabla$-covariant constancy of the elements of $\mathfrak{m}$ implies that the components of the Riemann tensor of $N_0$ along $\mathfrak{m}$ vanish and so the reduced holonomy of the Levi-Civita connection of $N_0$ reduces further from $SO(n_0)$ to $SO(\ell)$, where $\ell=n_0-\mathrm{dim}\, \mathfrak{m}$. Thus, $N_0$ is locally isometric to $N\times \bR^\ell$. A similar argument as that presented in the previous case establishes that $M$ is locally isometric and holomorphic  to $N\times K$, where $K$ is a KT group manifold and $N$ is a Calabi-Yau one. Notice that the decomposition of $K$ as $G\times\bR^\ell$ is isometric but not holomorphic.
If $M$ is simply connected and complete, then $M=N\times K$ globally.
\end{proof}

\begin{remark}
In the proof above we have used that the commutators of two $\h\nabla$-covariantly constant vector fields on a Riemannian manifold $(M,g,H)$ is given in terms of $H$, $dH=0$. This is well known, see \cite{uggp}. In particular, if $X$ and $Y$ are $\h\nabla$-covariantly constant vector fields, then $[X,Y]=(\iota_X \iota_Y H)^\flat$.  Note also that group manifolds with a complex structure are also referred to as Samelson spaces.
\end{remark}

\begin{remark}
The result of the above theorem for strong CYT manifolds has been demonstrated in \cite{qzfz, bfg, btp} under some assumptions -- in the last two references the theorem is stated under the assumption of compactness. Here, we assume  that $(M^{2k}, g, I, H)$ is of
NSZ-type. This leads to a simplification of that proof as compared to the more general result presented in \cite{qzfz}.  However, this additional assumption suffices for our purpose as it is  used again in the investigation $G_2$ and $\mathrm{Spin}(7)$ manifolds with torsion below.
\end{remark}

This result can be used to classify the CYT in 6 dimensions with $\h\nabla$-covariantly constant non-vanishing torsion $H$. For a more extensive list of examples, see \cite{btp}.

\begin{corollary}
Let $(M^6, g, I, H)$ be a simply connected complete strong CYT manifold with $\h\nabla$-covariantly constant torsion $H$ and $H\not=0$, i.e. $M^6$ is not Calabi-Yau. Then, $M^6$ is isometric to
\be
SU(2)\times SU(2)~,~~~SU(2)\times \bR^3~.
\ee
\end{corollary}
\begin{proof}
The compact simply connected semisimple groups $G$ up to dimension 6 are isomorphic to either $SU(2)$  or to $SU(2)\times SU(2)$.  In the latter case, there is a CYT structure on $M=K=SU(2)\times SU(2)$. This is equipped with the bi-invariant  metric $g$  and 3-form torsion $H$ -- the latter is constructed from the structure constants of the two $SU(2)$ subgroups. If the radii of the two spheres are equal, there is a left- (or right-) invariant  complex structure $I$ on $SU(2)\times SU(2)$ compatible with the metric and torsion, see also remark below.

If the semi-simple group $G$ is isomorphic to $SU(2)$, then $M$ is isometric to $SU(2)\times \bR^3$. In this case, $N=\bR^2$ and $K=SU(2)\times \bR$ -- the dimension of the span $\mathfrak{k}$ for an integrable complex structure is four. This is because the curvature of all 2-dimensional Calabi-Yau manifolds vanishes, i.e. $N=\bR^2$.  The metric is constructed from the bi-invariant metric on $SU(2)$ and the bi-invariant 3-form torsion $H$ is identified with the structure constants of $SU(2)$. There is a left- (or right-) invariant complex structure on $SU(2)\times \bR$ which can be expended to $SU(2)\times \bR^3$ that  is  also compatible with $g$ and $H$.
\end{proof}

A  general exploration  of compact, strong, 6-dimensional  CYT manifolds has been carried out in \cite{abls} without the assumption that $H$ is $\h\nabla$-covariantly constant. However, the classification has not been completed as of yet.

\begin{remark}
There are several CYT structures that can be put on $SU(2)\times SU(2)$, see e.g. \cite{gpheterotic}. In all cases, the Hermiticity of the metric requires that the radii of the two $SU(2)$ subspaces are equal. $SU(2)\times SU(2)$ is a semisimple  group manifold and so this case $M=K$ is a group.
Similarly, there are several CYT structures on $SU(2)\times \bR^3$, see  a description  of the left- and right- invariant complex structures on $SU(2)\times \bR$ in \cite{gpew}.

\end{remark}

\subsection{HKT manifolds with parallel 3-form torsion}

The results of the previous section for strong KT and strong CYT manifolds can be generalised to strong HKT manifolds.  This result has already been obtained for compact HKT manifolds in \cite{bfgv}. Here, we shall use Theorem \ref{th:gdec} for the proof.  In particular, we have the following:

\begin{theorem}\label{th:HKT}
 Let $(M^{4q}, g, I_r, H)$ is a strong HKT manifold with $H\not=0$  that satisfies the assumptions of Theorem \ref{th:gdec} and with a decomposition (\ref{thdcomp})   of NSZ-type, then $M$ locally isometrically  and tri-homorphically splits into a product $N\times K$, where $N$ is hyper-K\"ahler manifold and $K$ is a group manifold with an HKT structure. Moreover, if $M$ is simply connected and complete, then $M=N\times K$ globally.
\end{theorem}
\begin{proof}

The proof is similar to the one given for the CYT case in the Theorem \ref{th:KTCYT} as $(M^{4q}, g, I_r, H)$ is admits a CYT structure with respect to each of the complex structures $I_r$.  So,  we shall be brief. Let $\{I_r, r=1,2,3\}$ be the hyper-complex structure associated with the HKT structure on $M$, $I_1^2=I_2^2=-{\bf 1}$, $I_1I_2=-I_2 I_1$ and $I_3=I_1 I_2$. We have already demonstrated in Theorem \ref{th:KTCYT}, using the Lemma \ref{le:group},  that $M$ is locally isometric as a Riemannian manifold  to $N_0\times G$. Thus the tangent space at $p\in T_pM$, $p\in M$, decomposes as $T_p N_0\oplus T_p G$.
If $I_r (T_pG)\subseteq T_p G$ for $r=1,2$, then a similar argument to that produced for the KT case in Theorem \ref{th:KTCYT} implies that $M$ is locally isometric to
a group manifold $K=G$ with an HKT structure. Note that if $I_r (T_pG)\subseteq T_p G$ for $r=1,2$, then it follows that $I_3 (T_pG)\subseteq T_p G$.

On the other hand if $I_r T_pG\nsubseteq T_p G$ for $r=1,2$, we proceed as in the CYT case and  identify the Lie algebra of $G$, $\mathfrak{g}$, with the left-invariant vector fields $\{ L_a; a=1, \dots, \mathrm{dim}\, G\}$ on $G$. These are $\h\nabla$-covariantly constant. Then, consider the span
\be
\mathfrak{k}=\mathrm{Span}\,\langle L_a, I_1(L_a), I_2(L_a)\rangle_{ L_a\in\mathfrak{g}}~.
\ee
If all $I_r(L_a)\in \mathfrak{g}$, $r=1,2$,  then $I_r( T_p G)\subseteq T_p G$ and this case has already been investigated above. Thus, there must exist some vector fields $I_r(L_a)$ that are not contained in $\mathfrak{g}$. To continue, notice that all $I_r(L_a)$ are $\h\nabla$-covariantly constant as both $I_r$ and $L_a$ are $\h\nabla$-covariantly constant. As a result, their Lie algebra bracket can be expressed in terms of $H$.  Thus, $\mathfrak{k}$ can be identified with the Lie algebra of a group $K$ whose structure constants are given by $H$. In  this case, $K$ is not semisimple.  Writing,
\be
\mathfrak{k}=\mathfrak{g}\oplus\mathfrak{m}~,
\ee
as in the KT case,  the vector fields spanning $\mathfrak{m}$ are $\nabla$-covariantly constant and therefore they commute. Moreover, since $[\mathfrak{m}, \mathfrak{g}]=0$, as the components of $H$ along $\mathfrak{m}$ vanish, the decomposition above is compatible with the Lie algebra structure.  The integrability condition of the $\nabla$-covariant constancy of the elements of $\mathfrak{m}$ implies that the components of the Riemann tensor of $N_0$ along $\mathfrak{m}$ vanish and so the reduced holonomy of the Levi-Civita connection of $N_0$ reduces further from $SO(n_0)$ to a subgroup of $SO(\ell)$, where $\ell=n_0-\mathrm{dim}\, \mathfrak{m}$. Thus $N_0$ is locally isometric to $N\times \bR^\ell$. A similar argument as that presented in the CYT case establishes that $M$ is locally isometric and tri-holomorphic  to $N\times K$, where $K$ is an HKT group manifold and $N$ is a hyper-K\"ahler one.
If $M$ is simply connected and complete, then $M=N\times K$ globally.
\end{proof}

\begin{corollary}\label{cor:hktx}
Let $(M^8, g, I_r, H)$ be a  strong HKT manifold with $\h\nabla$-covariantly constant torsion $H$ and $H\not=0$, i.e. $M^8$ is not hyper-K\"ahler. Then, $M^8$ is locally isometric to $SU(3)$, $\bR^2\times SU(2)\times SU(2)$ and $N^4\times \bR\times SU(2)$, where $N^4$ is a hyper-K\"ahler manifold. Moreover, there is at least one HKT structure on these spaces such that $M^8$ locally tri-holomorphic to
\be
SU(3)~,~~~(\bR\times SU(2))\times (\bR\times SU(2))~,~~~N^4\times (\bR\times SU(2))~,
\ee
where $\bR\times SU(2)$  is an HKT manifold  and $N^4$ is a hyper-K\"ahler manifold.
\end{corollary}
\begin{proof}

There are three compact semisimple simply connected groups $G$ up to dimension 8, $SU(3)$, $SU(2)\times SU(2)$ and $SU(2)$.  It is known that $M=G=SU(3)$ admits an HKT structure \cite{Spindel}. This is explicitly constructed in section  \ref{sub:example} below.

For $G=SU(2)\times SU(2)$, $M^8$ is locally isometric to $\bR^2\times SU(2)\times SU(2)$. As $\bR\times SU(2)$  is an HKT manifold,  $M^8$ is locally triholomorphic to $(\bR\times SU(2))\times (\bR\times SU(2))$.

Finally for $G=SU(2)$, we have established that $M$ is locally isometric to $N^4\times \bR\times SU(2)$. Next, if $\mathrm{dim}\, \mathfrak{k}=4$, then $\bR\times SU(2)$ has an HKT structure and $M^8$ locally splits tri-holomorphically as  $N^4\times (\bR\times SU(2))$, where $N^4$ is a hyper-K\"ahler manifold. On the other hand if $\mathrm{dim}\, \mathfrak{k}=8$, then $M^8$ is locally isometric to $\bR^5\times SU(2)$.   As the subspace   $\bR\times SU(2)$ is an HKT manifold and $\bR^4$ is hyper-K\"ahler, there is an HKT structure such that $M^8$ is locally tri-holomorphic to $(\bR\times SU(2))\times \bR^4$.

\end{proof}

\begin{remark}
Under the assumption of compactness, the result of this theorem has been demonstrated in \cite{bfgv} without the assumption that $(M^{4q}, g, I_r, H)$ must be of
NSZ-type. It is likely that neither compactness nor the NSZ-type of restrictions are required for the proof of the theorem and can be replaced by weaker assumptions. 
\end{remark}

\begin{remark}

Compact examples of 8-dimensional HKT manifolds with $\h\nabla$-covariant constant torsion include
\be
SU(3)~,~~~(U(1)\times SU(2))\times (U(1)\times SU(2))~,~~~(U(1)\times SU(2))\times T^4~,~~~(U(1)\times SU(2))\times K_3~.
\ee
More compact examples can be constructed by identifying these spaces with a discrete subgroup of their isometry group that in addition preserves the hyper-complex structure, i.e. it preserves their HKT structure.  Note though that there are many non-compact 4-dimensional hyper-K\"ahler manifolds in addition to the compact ones $T^4$ and $K_3$.

It should be also noted that $U(1)\times SU(2)$ admits several HKT structures explored at length in \cite{gpew}. Thus $(U(1)\times SU(2))\times (U(1)\times SU(2))$ and $(U(1)\times SU(2))\times K_3$ also admit several HKT structures.

\end{remark}

\begin{remark}
We have not classified all HKT structures on $M^8$. This would have required to consider all possibilities. At every point on an 8-manifold there are
$SO(8)/Sp(2)$ (almost) hyper-complex structures compatible with a metric. Such a task is an algebraic problem in the present context as we are effectively looking for left invariant structures. Nevertheless, it is rather involved even after taking into account the obvious automorphisms of the Lie algebras involved to simplify the equations.
\end{remark}

\section{$G_2$ and $\mathrm{Spin}(7)$ manifolds with torsion}

\subsection{$G_2$ manifolds}

Let $(M^7, g, \varphi, H)$ be a manifold with torsion and closed 3-form $H$ such that the holonomy of $\h\nabla$ connections is contained in $G_2$ -- see \cite{tfsi2} for a description of the geometry.  For such manifolds, the fundamental $G_2$ form, $\varphi$, is $\h\nabla$-covariantly constant, $\h\nabla\varphi=0$.

All strong $G_2$ manifolds with torsion are steady generalised Ricci solitions, see definition \ref{def:rsol}. In fact, one can show that
\be
\h\R=\h\nabla \theta~,
\ee
where $\theta$ is the Lee form appropriately normalised. The normalisation of $\theta$  has no significance in the argument that follows. The Lee form $\theta$ can be expressed in terms of the torsion $H$ and the fundamental form $\varphi$. If, in addition,  $H$  is $\h\nabla$-covariantly constant, $\h\nabla \theta=0$, and so
$\h\R=0$. Then, the remark below the Lemma \ref{le:group} applies.  Therefore if $(M^7, g, \varphi, H)$ is of NSZ-type, then $M^7$ locally decomposes  as $N\times G$, where $N$ is a Riemannian manifold with $H\vert_N=0$ and $G$ is a compact semisimple Lie group equipped with a parallelising connection.  For the proof of the theorem below, the techniques adopted for CYT manifolds have to be adapted as the fundamental form $\varphi$ of the $G_2$ structure is not an endomorphism of the tangent space. Instead, it defines a cross product on the tangent space.

\begin{theorem}\label{cor:g2}
Let $(M^7, g, H)$ be a strong $G_2$ manifold with torsion the closed 3-form $H\not=0$ that satisfies the assumptions of Theorem \ref{th:gdec}  and with a decomposition (\ref{thdcomp})   of NSZ-type.  Then, $M^7$ is locally isometric to
 $\bR\times SU(2)\times SU(2)$ or to $N^4\times SU(2)$, where $N^4$ is a hyper-K\"ahler manifold. Moreover, if $M^7$ is complete and simply connected, then $M^7$ is globally isometric to these manifolds.
\end{theorem}

\begin{proof}

We have already demonstrated that $M^7$ is locally isometric to $N\times G$ with the torsion $H$ to have support only on $G$, where $G$ is a semisimple group. As mentioned, this is a consequence of $\h\R=0$, Lemma \ref{le:group} and the assumption that the decomposition (\ref{thdcomp}) is of NSZ-type.
  The simply connected  semisimple groups up to dimension $7$  are $SU(2)$ and $SU(2)\times SU(2)$ and so $G$ can be identified with one of these. If  $G=SU(2)\times SU(2)$, then $N=\bR$. Thus $M$ must be locally isometric to  $\bR\times SU(2)\times SU(2)$.  As it will be explicitly explained  in the remarks below, this  space admits a strong $G_2$ structures with torsion.

It remains to investigate the  $G=SU(2)$ case.  There are two possibilities to be considered depending on whether or not the tangent space of $SU(2)$ at each point is an associative 3-plane.  It is sufficient for the tangent space at the origin of $SU(2)$ to be associative as it is spanned by the left invariant vector fields $\{L_r: r=1,2,3\}$ on $SU(2)$ and $\varphi$ is $\h\nabla$-covariantly constant. Then, it will be associative at every point as the structure is preserved by parallel transport.

To begin, suppose that the tangent space of $SU(2)$ is associative.  This means that $\iota_{L_r}\iota_{L_s}\varphi\in T_p^*SU(2)$ for every $p\in M^7$.  In particular, $(\iota_{L_r}\iota_{L_s}\varphi)_{N^4}=0$.  Then, the fundamental form $\varphi$ can be  written, up to an overall constant scale,  as
\be
\varphi=\lambda^1\wedge \lambda^2\wedge\lambda^3+\lambda^r\wedge \omega_r~,
\label{g2f1x}
\ee
where $\{\lambda^r: r=1,2,3\}$ is the left invariant co-frame on $SU(2)$ associated to $\{L_r: r=1,2,3\}$ and $\{\omega_r:r=1,2,3\}$ are 2-forms with $\iota_{L_s}\omega_r=0$. Since $\varphi$ is $\h\nabla$-covariantly constant and the same applies for $\lambda^r$, $\omega_r$ are $\nabla$-covariantly constant and their domain is $N^4$. On the assumptions made about the associative property of $SU(2)$, the expression for $\varphi$ in (\ref{g2f1x}) is the most general one apart for not having components along $N^4$, i.e. it does not have a non-vanishing term $\varphi_{N^4}$. To see why this is vanishing, note that the isotropy group of an associative 3-plane in $G_2$ is $SO(4)$. As $\varphi$ is $G_2$ invariant, it is also $SO(4)$ invariant. But if there were a non-vanishing component $\varphi_{N^4}$ of $\varphi$, it would have reduced the isotropy group of the associative 3-plane to a subgroup of $SO(4)$.  This is a contradiction.  Therefore, such a component should vanish $\varphi_{N^4}=0$. This argument has been used before in \cite{bryant, bryantsalamon} for $G_2$ manifolds.

The forms $\omega_r$ are further restricted. This is because for $\varphi$ to be the fundamental form of a $G_2$ structure, it has also to satisfy the positivity condition established by Cartan and put in the form below in \cite{bryant}.  This means that the quadratic form
\be
B(X, Y)\,d\mathrm{vol} (M^7)=\frac{1}{6}\, \iota_X\varphi\wedge \iota_Y\varphi\wedge \varphi~,
\ee
must be positive definite depending on the choice of orientation of $M^7$. Applying this criterion to (\ref{g2f1x}), it  implies that $\omega_r$ are either self-dual or anti-self-dual forms on $N^4$. In either case, $N^4$ must be a hyper-K\"ahler manifold as in addition these three 2-forms are $\nabla^{N^4}$-covariantly constant. Thus, $M^7$ is locally isometric to $N^4\times SU(2)$, where $N^4$ is a hyper-K\"ahler manifold. Such manifolds   admit a $G_2$ structure with torsion.

Alternatively, suppose that $\iota_{L_r}\iota_{L_s}\varphi_{N^4}\not=0$.  In such a case $\iota_{L_r}\iota_{L_s}\varphi_{N^4}$ is $\nabla^{N^4}$-covariantly constant. Therefore $N^4$ is locally isometric $\bR\times \tilde N^3$. However, as we have explained, all $G_2$ manifolds with torsion that satisfy the assumption of the theorem must have  $\h\R=0$. Applying this to the case at hand, the Ricci tensor  of $\tilde N^3$ vanishes, $\R^{\tilde N^3}=0$. As for 3-dimensional manifolds the curvature is expressed in terms of the Ricci tensor, $N^3$ is flat. Therefore, $M^7$ must be locally isometric to $\bR^4\times SU(2)$. As $\bR^4$ is a hyper-K\"ahler manifold, this case  is included  in the previous case where the tangent space of $SU(2)$ was associative. Moreover, if $M^7$ is simply connected and complete, then $M^7=SU(2)\times SU(2)\times \bR$ or $M^7=N^4\times SU(2)$ globally as a consequence of the de Rham theorem.
\end{proof}

\begin{remark}
As the only 4-dimensional compact hyper-K\"ahler manifolds are $T^4$ and $K_3$, compact strong $G_2$ manifolds include  $U(1)\times SU(2)\times SU(2)$, $T^4\times SU(2)$ and $K_3\times SU(2)$. More compact examples can be  constructed from those after an identification with a discrete group of isometries that preserves the $G_2$ structure.
\end{remark}

\begin{remark}
The $G_2$ structure on $N^4\times SU(2)$, $U(1)\times SU(2)\times SU(2)$ and $T^4\times SU(2)$ is not unique even after fixing the orientation of $M^7$. We shall not go too deep into detail. The argument is similar to that of \cite{gpew} that constructed four possible HKT structures on $U(1)\times SU(2)$ separated into two pairs by the choice of orientation.   A $G_2$ structure on $U(1)\times SU(2)\times SU(2)$ can be constructed using one of the HKT structures on the subspace $U(1)\times SU(2)$  of $U(1)\times SU(2)\times SU(2)$ viewed as an HKT manifold. If the Hermitian forms of $U(1)\times SU(2)$ are $\omega_r$, that can be either left- or right- invariant, then
\be
\varphi=\lambda^1\wedge \lambda^2\wedge\lambda^3+\lambda^r\wedge \omega_r~,
\label{g2f2x}
\ee
where $\{\lambda^r: r=1,2,3\}$ is either the left- or right- invariant co-frame on the remaining $SU(2)$ subspace of $U(1)\times SU(2)\times SU(2)$.
A similar result can be obtained for $T^4\times SU(2)$ as well as their universal covers $\bR\times SU(2)\times SU(2)$ and $\bR^4\times SU(2)$.

\end{remark}

\subsection{$\mathrm{Spin}(7)$}

Let $(M^8, g, H)$ be a manifold with torsion the closed 3-form $H$ such that the holonomy of $\h\nabla$ connection is contained in $\mathrm{Spin}(7)$ -- see \cite{si2} for a description of the geometry.
Given an orthonormal co-frame $\{ e^0, e^a; a=1,\dots, 7\}$ on $M^8$, a normal form of the fundamental  $\mathrm{Spin}(7)$ 4-form $\phi$, or Cayley form,  is
\be
\phi=e^0\wedge \varphi+{}*\varphi~,
\label{fspin7form}
\ee
where $\varphi$ is a  fundamental $G_2$ form and ${}^*\varphi$ is its dual.  In an adapted orthonormal co-frame $\varphi=e^{123}+ e^{145}+e^{167}+ e^{274}+e^{256}+e^{367}+e^{345}$, where $e^{123}=e^1\wedge e^2\wedge e^3$ and similarly for the rest of the components.

\begin{theorem}\label{cor:spin7}
Let $(M^8, g, \phi, H)$ be a strong $\mathrm{Spin}(7)$ manifold with torsion  that satisfies the assumptions of Theorem \ref{th:gdec} and with a decomposition (\ref{thdcomp})   of NSZ-type.  Then, $M^8$ is locally isometric to
$SU(3)$, $\bR^2\times SU(2)\times SU(2)$ or to $\bR\times N^4 \times SU(2)$, where $N^4$ is a hyper-K\"ahler manifold. Moreover, if $M^8$ is complete and simply connected, $M^8$ is globally isometric to these manifolds.
\end{theorem}
\begin{proof}

All strong $\mathrm{Spin}(7)$ manifolds with torsion are steady generalised Ricci solitons. One can show, as in the $G_2$ case previously, that $\h\R=\h\nabla\theta$, where $\theta$ is the Lee form. $\theta$ can be expressed in terms of $H$ and the fundamental form $\phi$. As $M^8$ satisfies the assumptions of Theorem \ref{th:gdec}, $H$ is $\h\nabla$-covariantly constant and so $\h\R=0$. Then, the results of Lemma \ref{le:group} can be adapted to this case. This together with the assumption that $M^8$ is of NSZ-type in the decomposition (\ref{thdcomp}), one concludes that $M^8$ can locally be decomposed as $N\times G$, where $N$ is a Riemannian manifold with $H\vert_N=0$ and $G$ is a compact semi-simple group equipped with a parallelising connection.
 The semisimple groups up to dimension 8 are $SU(3)$, $SU(2)\times SU(2)$ and $SU(2)$.

It is straightforward to construct a left-invariant  $\mathrm{Spin}(7)$ fundamental 4-form on $SU(3)$. For this, we write $\phi$ in a adapted orthonormal co-frame and then  replace it   with the left-invariant co-frame on $SU(3)$. This proves that $SU(3)$ admits a $\mathrm{Spin}(7)$ structure. But of course the holonomy of the $\h\nabla$ connection is trivial as the manifold is parallelisable.

The next possibility is that $M^8$ is locally isometric to $N^2\times SU(2)\times SU(2)$. To continue, we use that the Ricci tensor $\h R$ of $M^8$ vanishes as established above.  This implies that the Ricci tensor $\R$ of $N^2$  vanishes as well.  The Riemann tensor of a 2-dimensional manifold can be expressed in terms of the Ricci tensor.   As a consequence of this,  $N^2$ is flat and so $M^8$ is locally isometric to $\bR^2\times SU(2)\times SU(2)$.  Writing $\bR^2\times SU(2)\times \ SU(2)$ as $\bR\times \big(\bR\times SU(2)\times SU(2)\big)$, a $\mathrm{Spin}(7)$ structure can be constructed on this space by using (\ref{fspin7form}) and  utilizing the $G_2$ structure on $ \bR\times SU(2)\times  SU(2)$ as described in (\ref{g2f2x}), where $e^0=dv$ and $v$ is the coordinate of the first $\bR$ subspace.

It remains to investigate the case that $M^8$ is locally isometric to $N^5\times SU(2)$. Choosing an  orthonormal basis on $SU(2)$  of a left invariant vector fields $\{L_r: r=1,2,3\}$,  it is known that $\iota_{L_1} \iota_{L_2}\iota_{L_3} \phi\not=0$. As the triple product defined by $\phi$ is compatible with the metric $g$ and $\{L_r: r=1,2,3\}$ is an orthonormal basis, then $\iota_{L_1} \iota_{L_2}\iota_{L_3} \phi_p\in T^*_p N^5\subset T_pM^8$, for $p\in M^8$ and moreover it has length 1.  This follows from the general properties of non-degenerate triple products and in particular that defined  by the Cayley form, see e.g. eqn (14) in \cite{krasnov} for the equivalent properties of triple products on vector spaces and  in \cite{salamon3}. Setting $e^0\equiv \iota_{L_1} \iota_{L_2}\iota_{L_3}\phi$, $e^0$ is $\nabla^{N^5}$-covariantly constant and as a result $N^5$ locally splits to $\bR\times \tilde N^4$. Furthermore, $\phi$ can be written as $e^0\wedge \varphi+{}^{*_7}\varphi$,  where $\varphi$ is a
 fundamental $G_2$ form on $ \tilde N^4\times SU(2)$ and the duality operation is taken on $ \tilde N^4\times SU(2)$; see e.g. Theorem 6 in \cite{krasnov} for the linear algebra equivalent statement.  Then, our theorem follows by applying the argument made in Theorem \ref{cor:g2} on the manifold $\tilde N^4 \times SU(2)$ after we consider it now as a $G_2$ manifold with fundamental form $\varphi$.
\end{proof}

\begin{remark}
 As the only  4-dimensional compact hyper-K\"ahler manifolds are $T^4$ and $K_3$, compact strong $\mathrm{Spin}(7)$ manifolds include $SU(3)$, $T^2\times SU(2)\times SU(2)$, $T^5\times SU(2)$ and $S^1\times K_3\times  SU(2)$. Many more compact examples   are constructed from those after an identification of the above geometries with a discrete group of isometries that preserves the $\mathrm{Spin}(7)$ structure. The $\mathrm{Spin}(7)$ structure with torsion on $T^2\times SU(2)\times SU(2)\times T^2$ and  $T^5\times SU(2)$ is not unique and several can be constructed in a way similar to that described in more detail in the $G_2$ case above.
\end{remark}

\subsection{Geometric structures and restrictions on torsion}

Before we proceed further, let us summarise the outcome of the results we have obtained so far. It is clear  that to construct compact strong CYT, HKT, $G_2$ and $\mathrm{Spin}(7)$ geometries with torsion, which  that are not locally products $N\times G$ with  $G$ a group,  the restriction that we have imposed to $H$ to be  $\h\nabla$-parallel must be either  weaken or removed altogether from the assumptions. In fact, this point could have been made from the very beginning for general manifolds with torsion a $3$-form as a consequence of Theorem \ref{th:gdec}.

There are two alternative strategies to weaken the assumptions that have been made so far to investigate the geometries with torsion a $3$-form.  Both are related to the investigation of generalised  steady Ricci solitons. These are geometries
that arise in the investigation of the properties of  generalised  Ricci flow.  Here, we shall comment on one of them. The other is used to show that generalised steady Ricci solitons admit certain symmetries that can be utilised to determine their structure -- see the method employed to determine the structure of $8$-dimensional HKT manifolds in the second part of the paper. To begin, let us recall the definition of generalised steady Ricci solitons, see \cite{GFS} for a review.

\begin{definition}\label{def:rsol}
The manifold $(M, g, H)$ is a generalised steady Ricci soliton, iff  it  satisfies the equation
\be
\h \R=\h\nabla V^\flat~,
\label{steadycon}
\ee
for some vector field $V$.
\end{definition}

\begin{remark}

 A consequence of the Bianchi identity in (\ref{bianchi}) and the restrictions on the holonomy of $\h\nabla$ is that  strong CYT, HKT, $G_2$ and $\mathrm{Spin}(7)$ manifolds are generalised steady Ricci solitons.  In such a case, $V^\flat$ is identified, after an appropriate numerical normalisation,  with the Lee form, $\theta$,  of these geometries. The Lee form is defined as $\theta=c\, {}^*(\phi\wedge {}^*\delta \phi)$, where $\phi$ is the fundamental form of the associated geometry, e.g. $\phi$ is the Hermitian form for the CYT geometry, $\delta$ is the adjoint of the exterior derivative $d$ and $c$ is an appropriately chosen constant.  In particular, after an appropriate choice of $c$,  it can be shown that  these geometries satisfy
\be
\h \R=\h\nabla \theta~.
\label{rtheta}
\ee
Therefore, all these geometries are  generalised steady Ricci solitons.
\end{remark}

 To proceed a  useful formula is the following:
\begin{prop}\label{prop:gsrs}
Let $(M^n, g, H)$ be a generalised steady Ricci soliton ($dH=0$), i.e. it  satisfies (\ref{steadycon}). Then
\be
\nabla^2 \Big(-\frac{1}{2} \h R+\frac{1}{12} H^2\Big)= \nabla_V \Big(\frac{1}{2} \h R+\frac{1}{12} H^2\Big)+\h \R^2 ~.
\label{steadyf}
\ee
Therefore, if  $M^n$ is compact and one of the functions $R$, $\h R$ or $H^2$ is constant, then the Ricci tensor $\h \R$ vanishes and $V$ is $\h\nabla$-covariantly constant.
\end{prop}

The proof of this formula and the statement of the proposition above was given in \cite{gpew}  -- in fact a more general formula was demonstrated there that applies to both expanding and shrinking generalised Ricci solitons.

\begin{remark}
Note that the formula (\ref{steadyf}) does not only apply to manifolds with a special or exceptional geometric structure but to all manifolds that satisfy the generalised steady soliton condition (\ref{steadycon}). A similar statement has been made in \cite{sins} for gradient generalised Ricci solitons using a formula proven in Proposition 4.33 in \cite{GFS}, i.e. for generalised Ricci solitons that satisfy (\ref{steadyf}) with $V^\flat=-2d\Phi$ with $\Phi$ a function on $M^n$. It also generalises some results that have been previously proven  for  $G_2$ manifolds with torsion, see formula (7.78) in  \cite{sins2} that was demonstrated using details of the $G_2$ geometry of these manifolds.
\end{remark}

The above proposition suggests that one option to weaken the $\h\nabla$-covariant constancy condition on $H$ is to replace it with requiring that $H^2$ is constant. Clearly if $H$ is $\h\nabla$-covariantly constant, then $H^2$ is constant but the converse is not necessary true. For compact manifolds without boundary, the constancy of $H^2$  will imply from (\ref{steadyf})  that $R$ is constant and $\h \R=0$. In turn $\h \R=0$ gives that  $H$ is co-closed. Therefore, as $H$ is also closed, $H$ is harmonic. Thus for $H\not=0$, the third de Rham cohomology of $M^n$ should not vanish, $H^3_{\mathrm {dR}}(M^n)\not=0$. This is a topological condition on $M^n$.

However, the constancy condition on $H^2$ and the $\nabla$-covariant constancy of $H$ are not always independent, especially for harmonic forms. It is a consequence of the Bochner-Weitzenb\"ock formula that under certain conditions harmonic forms either  do not exist or  are $\nabla$-covariantly constant. For a 3-form $H$, this formula reads as
\be
(d\delta+\delta d) H=-\nabla^2 H+\mathcal {R} (H)~,
\label{bwf}
\ee
where $\delta$ is the adjoint of $d$ and
\be
\mathcal {R} (H)_{i_1i_2i_3}= \R_{i_1}{}^k H_{i_2i_3 k}-2 R_{i_1}{}^k{}_{i_2}{}^m H_{i_3 km}+\mathrm{cyclic}(i_1, i_2, i_3)~.
\ee
Taking the inner product of (\ref{bwf}) with $H$  and integrating over $M^n$, one finds that for a Harmonic form $H$, $dH=\delta H=0$, to be $\nabla$-covariantly constant, the integral of the inner product $(H, \mathcal {R} (H))$ over $M^n$ must vanish, where $M^n$ is taken to be compact without boundary. In addition, if the integral of $(H, \mathcal {R} (H))$ over $M^n$ is positive, then there are no harmonic 3-forms on $M^8$.  This is an additional restriction on the geometry of $M^8$ that  limits  the range of examples that can be constructed.  These results can be summarised as follows:

\begin{prop}\label{prop:Hdec}
Let $(M^n, g, H)$ be a compact  generalised steady Ricci solition with the (pointwise) length of $H$, $|H|$, constant or alternatively the Ricci scalar $R(g)$ of $g$ constant. Then, $\h\R=0$.  If, in addition, the integral of $(H, \mathcal {R} (H))$ over $M$ vanishes, then $M$ locally decomposes as a Riemannian manifold into a product $\times_\alpha N_\alpha$, where $\alpha$ labels the distinct eigenvalues of a symmetric bilinear $h$ constructed from $H$.  Moreover, if $M$ is simply connected and complete, then $M=\times_\alpha N_\alpha$ globally.
 \end{prop}

\section{Generalised Ricci solitons and   compact strong HKT manifolds}

\subsection{Generalised Ricci solitons and strong CYT structures}

Before, we prove a rigidity theorem for  compact strong HKT 8-manifolds, for completeness, let us summarise some results for strong CYT manifolds in relation to generalised Ricci solitons. The most promising avenue to pursue  examples of compact strong CYT and HKT   manifolds  (and also  $G_2$ and $\mathrm{Spin}(7)$ manifolds with torsion), especially for examples with low dimension,  is to use the generalisation of the
Perelman's theorem  as adapted  for the generalised Ricci flows in \cite{GFS}.  This states that all compact  generalised steady Ricci solitons are  generalised gradient Ricci solitons.
As we have already mentioned,  generalised steady Ricci solitons $(M^n, g, H)$ satisfy the equation (\ref{steadycon}),
while gradient generalised Ricci solitons satisfy the equation
 \be
\h{\mathrm{Ric}}_{ij}=-2\h\nabla_i\partial_j\Phi~,
\ee
where $\Phi$ is a function (dilaton) on $M^n$.  Therefore, the above theorem implies that there must exist  a $\h\nabla$-covariantly constant vector field $Y$, $\h\nabla Y=0$, such that
\be
Y=V+2d\Phi^\flat~,~~~\h\nabla Y=0~.
\ee
As a result, $Y$ is Killing and $\mathcal{L}_Y H=0$ for $H$ a closed 3-form.

It has already been mentioned that  manifolds with a connection $\h\nabla$ whose holonomy is included in $SU(k)$, $Sp(q)$, $G_2$ and $\mathrm{Spin}(7)$ are steady generalised Ricci solitons with $V=\theta^\flat$, where $\theta$ is the Lee form.  This is because they satisfy (\ref{rtheta}) as  a consequence that first Bianchi identity in (\ref{bianchi}) and the restriction on the holonomy of $\h\nabla$. If the manifold is compact, as stated above, then these geometries are also gradient generalised Ricci solitons. This means that the vector field
\be
Y=\theta^\flat+2d\Phi^\flat~,
\ee
is $\h\nabla$-covariantly constant.  Thus for $Y\not=0$, these manifolds admit a Killing vector field that in addition $\iota_Y H=dY^\flat$. The last condition implies that $\mathcal{L}_Y H=0$ as $H$ is closed. This can be exploited to construct examples.

 Let $(M^{2k}, g, H, I)$  be a compact manifold with a strong CYT structure. Then,  either $Y=0$ or $M^{2k}$ admits a second $\h\nabla$-covariantly constant vector field $-IY$ in addition to $Y$.  If $Y=0$, and so $\theta=-2 d\Phi$,  then $M^{2k}$ is conformally balanced.    It is known that compact conformally balanced strong CYT manifolds must  be Calabi-Yau with  $H=0$ and $\Phi$  constant \cite{sigp1}. However, if $Y\not=0$, $M^{2k}$  admits  two Killing vector fields that leave $H$ invariant. The function $\Phi$ is also invariant, $Y\Phi=(-IY) \Phi=0$. It is clear that the holonomy of $\h\nabla$ reduces from $SU(k)$ to $SU(k-1)$. In fact, it turns out that both $Y$ and $-IY$ are holomorphic and commute. This has been used in \cite{abls} to prove some rigidity properties for 6-dimensional strong CYT manifolds.  However, the extension of such rigidity results beyond six dimensions remains an open question.

 \subsection{Principal bundles and strong HKT structures}

 Before we proceed to investigate the geometry of 8-dimensional HKT manifolds, we shall summarise in the Theorem \ref{th:pfibre} below the results of \cite{jggpa, jggpb, gpheterotic} on the geometry of HKT manifolds $(M^{4q}, g, H, I_r)$ that admit an action by a compact group $K$ such that the associated  vector fields $\{V_\alpha: \alpha=1, \dots, \mathrm{dim} K\}$ of the action are $\h\nabla$-covariantly constant and preserve the span of the hyper-complex structure $\{I_r:r=1,2,3\}$. As the vector fields $V$ are $\h\nabla$-covariantly constant, they do not have fixed points and the action of $K$ is almost free. For simplicity, let us assume that $K$ acts freely on $M^{4q}$ and so $(M^{4q}, B^{4(q-k)}, K)$ is a principal bundle
 with fibre group $K$ and base space $B^{4(q-k)}$ -- as $K$ preserves the span of the hyper-complex structure, it must have a dimension multiple of four. For a recent summary of the properties of HKT manifolds, see \cite{gpew}.

 \begin{theorem}\label{th:pfibre}
  Let $(M^{4q}, g, H, I_r)$ be a (not necessarily compact) manifold with a strong HKT structure that admits a free action by a compact Lie group $K$  on $M^{4q}$ that generates $\h\nabla$-covariant vector fields $\{V_\alpha: \alpha=1, \dots, \mathrm{dim}\, K=4k\}$ and preserves the span of the hyper-complex structure $\{I_r:r=1,2,3\}$. Then,
  $M^{4q}$ admits a principal bundle connection $\lambda^\alpha(X)\equiv g(X_\alpha, X)$ such that the metric $g$,  torsion $H$ and Hermitian structures $\{\omega_r : r=1,2,3\}$, with $\omega_r(X,W)=g(X, I_r W)$ decompose as
  \begin{align}
 &g=h^K_{\alpha\beta} \lambda^\alpha\otimes \lambda^\beta+ g^\perp~,~~~H=\mathrm{CS}(\lambda)+H^\perp~,
 \cr
 &\omega_r=\frac{1}{2}\,(\omega_r)^K_{\alpha\beta} \lambda^\alpha\wedge \lambda^\beta+\omega_r^\perp~,
 \label{hktdec}
 \end{align}
  where $\mathrm{CS}(\lambda)$ is the Chern-Simons 3-form of $\lambda$, $h^K_{\alpha\beta}=g(V_\alpha, V_\beta)$ and $g^\perp, H^\perp, \omega_r^\perp$ are the (horizontal) components of $g, H, \omega_r$ perpendicular  to the orbits of $K$, respectively. As $g^\perp$ and $H^\perp$ are invariant under the action of $K$ and transversal to the orbits, they are the pull-back of a metric $g^B$ and a 3-form $H^B$ of the base space $B^{4(q-k)}$ with the projecting map $\pi$.

  The group $K$ has structure constants $H^K$, $h^K(V_\alpha,  H^K(V_\beta, V_\gamma)) =H(V_\alpha, V_\beta, V_\gamma)=H_{\alpha\beta\gamma}$, and admits a left-invariant HKT structure associated with $(h^K, H^K, \omega_r^K)$, where both $h^K$ and $H^K$ are bi-invariant.

  Furthermore, as $\mathcal{L}_{V_\alpha} I_r=(\mathrm{B}_\alpha)^s{}_r I_s$, then $\{\mathrm{B}_\alpha: \alpha=1, \dots, 4k\}$ is a 3-dimensional representation of $K$,
  $[\mathrm{B}_\alpha, \mathrm{B}_\beta]=-\tilde H^K_{\alpha\beta}{}^\gamma \mathrm{B}_\gamma$ that in addition  satisfies $(I^K_r)^\beta{}_\alpha (\mathrm{B}_\beta)^s{}_r=-\epsilon_r{}^s{}_t (\mathrm{B}_\alpha)^t{}_r$ (no summation over $r$),  where $[V_\alpha, V_\beta]=- H^K_{\alpha\beta}{}^\gamma V_\gamma$.

  In addition, the curvature $\mathcal{F}$  of the connection $\lambda$ of the principal bundle $M^{4q}$ satisfies
  \be
  -\mathcal{F}_{\alpha ca} (I_r^\perp)^c{}_b+\mathcal{F}_{\alpha cb} (I_r^\perp)^c{}_a= (\mathrm{B}_\alpha)^s{}_r I^\perp_{sab}~
  \label{frestrict}
  \ee
  where an orthonormal co-frame  $\{e^a: a=1,\dots 4q-4k\}$ is chosen such that $g^\perp=\delta_{ab} e^a e^b$ and $\omega^\perp_r=\frac{1}{2} \omega^\perp_{rab} e^a\wedge e^b$.  The closure of $H$ implies that
  \be
   d H^B+h^K_{\alpha\beta} \mathcal{F}^\alpha\wedge \mathcal {F}^\beta=0~,
  \label{closx}
  \ee
 i.e. the Pontryagin class of the principal bundle $M^{4q}(B^{4(q-k)}, K)$ is trivial.
If $\mathrm{B}$ is the trivial representation, then the base space $B^{4(q-k)}$ is an HKT manifold otherwise $B^{4(q-k)}$ has a QKT structure. In the latter case, the associated bundle of $M^{4q}$ with respect to the representation $\mathrm{B}$ is identified with the vector bundle on $B^{4(q-k)}$ spanned by the three complex structures as a QKT manifold.
\end{theorem}
\hfill $\square $

\begin{remark}
Note that the components of $h^K, H^K, \omega^K$ and $B$ are constant in the co-frame $\{\lambda^\alpha, e^a: \alpha=1, \dots, \mathrm{dim} K; a=1,\dots, 4q-\mathrm{dim} K\}$  as $V_\alpha$ are $\h\nabla$-covariantly constant.  The curvature $\mathcal{F}$ as a 2-form has only horizontal components and (\ref{frestrict}) implies that it can be written  as
\be
\mathcal{F}=\mathcal{F}_{\mathfrak{sp}(q-1)}+\mathcal{F}_{\mathfrak{sp}(1)}~,
\ee
where we have used the decomposition of the space of 2-forms $\Lambda^2(\bR^{4q-4})$ under the $Sp(q-1)\cdot Sp(1)$ subgroup of $SO(4q-4)$.
The condition (\ref{frestrict}) restricts only the component $\mathcal{F}_{\mathfrak{sp}(1)}$. The component   $\mathcal{F}_{\mathfrak{sp}(q-1)}$ is $(1,1)$-form with respect to all complex structures and so the left hand side of the equation (\ref{frestrict}) vanishes when restricted on $\mathcal{F}_{\mathfrak{sp}(q-1)}$.
A more detailed discussion of the theorem and a proof can be found in \cite{gpheterotic}, where the CYT geometries are also treated.
\end{remark}

\begin{remark}

As the holonomy of $\h\nabla$ connection of the HKT manifold $(M^{4q}, g, H, I_r)$ is included in $Sp(q)$, the Ricci forms
\begin{equation}
\h\rho^r\equiv \frac{1}{4} \h R_{ij}{}^k{}_m (I_r)^m{}_k\, e^i\wedge e^j~,
\end{equation}
must vanish, $\h\rho^r=0$. Using eqn (4.25) in \cite{gpheterotic}, this gives a relation between the Ricci forms of the base space and the curvature $\mathcal{F}$.  In particular, one has that
\begin{equation}
(\h \rho^B)^r= Q^r_\alpha \mathcal{F}^\alpha~,~~~Q^r_\alpha\equiv h^K_{\alpha\beta}\, (\omega^B, \mathcal{F}^\beta)~.
\label{rhorf}
\end{equation}
This relates  the curvature of the base space and the curvature of the principal bundle. Observe also that in $Q$ only the $\mathcal{F}_{\mathfrak{sp}(1)}$ component of the curvature contributes.

It turns out that the converse is true. Suppose that $M^{4q}$ is constructed as a principal bundle, as described in Theorem   \ref{th:pfibre}, with fibre a group $K$ with an HKT structure and base space $B^{4(q-k)}$ that admits a QKT structure. Furthermore,  if the principal bundle connection $\lambda$ satisfies (\ref{frestrict}), the equation (\ref{closx}), related to the closure of $H$, holds and moreover (\ref{rhorf}) is also valid, then $M^{4q}$ admits a strong HKT structure.

\end{remark}

\begin{remark}
If the group $K$ does not act freely on $M^{4q}$ but its  action generates $\h\nabla$-covariantly constant vector field on $M^{4q}$ or  a Lie algebra action on $M^{4q}$ is generated by  $\h\nabla$-covariantly constant vector field with closed orbits, generically,  the base space $B^{4(q-k)}$  will be an orbifold.   The relation of geometries with skew-symmetric torsion that admit a group action to those of principal bundles has been known for sometime. For example, it has been observed in the classification of geometries of heterotic backgrounds in \cite{uggp} and reviewed in \cite{ugjggpx}.

\end{remark}

\begin{remark}

The relation between HKT and QKT manifolds has also been known for sometime. The first such examples have been in the context of homogeneous spaces investigated in \cite{op}. This followed the definition QKT manifolds and the investigation of their twistor spaces   in \cite{hopqkt}.  Further properties of QKT manifolds have been explored in \cite{siqkt}.

\end{remark}

 It is known that all three Lee forms of an HKT manifold $(M^{4q}, g, H, I_r)$ with respect to each of the complex structures $I_r$ are equal, so we set $\theta=\theta_1=\theta_2=\theta_3$. Strong HKT $(M^{4q}, g, H, I_r)$ manifolds are steady generalised Ricci solitons. If, in addition, $M^{4q}$ is  a gradient soliton, $\h{\R}=-2\h\nabla\Phi$, then there is an $\h\nabla$-covariant constant vector field $Y$ such that
 \be
 Y^\flat=\theta+2d\Phi~.
 \ee
 Such a manifold will admit another three $\h\nabla$-covariantly constant vector field $Y_r=- I_r Y$. It turns out that $Y$  is tri-holomorphic $\mathcal{L}_Y I_r=0$ and commutes with $Y_r$, $[Y, Y_r]=0$. While $Y_r$ are holomorphic with only the $I_r$ complex structure.  Assuming that the algebra of $\{Y_\alpha: \alpha=0,1, \dots, 3\}=\{Y, Y_r: r=1,2,3\}$ closes, it must be isomorphic to either $\oplus^4\mathfrak{u}(1)$ or to $\mathfrak{u}(1)\oplus \mathfrak{su}(2)$ -- this follows from the classification of compact Lie algebras of low dimension, i.e. those that admit a bi-invariant positive definite inner product.  In the latter case, the $\mathfrak{u}(1)$ subalgebra will be spanned by $Y$ while $\mathfrak{su}(2)$ will be spanned by the rest of vector field.

 Compact HKT manifolds $M^{4q}$ are also   generalised gradient Ricci solitons, i.e. they satisfy the equation
 \be
 \h {\R}+2\h\nabla\partial \Phi=0~,
 \ee
 for some scalar function $\Phi$.  This follows from the work of Perelman on Ricci solitons as generalised in \cite{GFS} for generalised Ricci solitons. Therefore, they admit $\h\nabla$-covariantly constant vector fields, as described above, that leave invariant both $g$ and $H$. Moreover, $\Phi$ will also be  invariant, $Y_\alpha(\Phi)=0$. Theorem \ref{th:pfibre}   can be adapted to HKT manifolds that are generalised gradient solitons, see also \cite{jggpa, jggpb},  as follows:

\begin{theorem}\label{th:pfibre2}
Let  the  HKT manifold $(M^{4q}, g, H, I_r)$  be  a gradient generalised Ricci soliton that  satisfies the assumptions of Theorem \ref{th:pfibre}.  Moreover, assume  that the $K$ group generates the vector field $\{Y_\alpha: \alpha=0,\dots ,3\}=\{Y, Y_r=-I_r Y: r=1,2,3\}$ on $M^{4q}$,  $Y^\flat=\theta+2d\Phi$ and $Y_\alpha\Phi=0$. Then, the metric $g$, torsion $H$ and Hermitian forms $\omega_r$ decompose as in (\ref{hktdec}), where $\lambda^\alpha(X)=g(Y_\alpha, X)$.

If $K$ is abelian, then  curvature of the principal bundle $\{\mathcal{F}^\alpha: \alpha=0,\dots, 3\}$ will be a $(1,1)$-form with respect to all complex structures, the base space $B^{4q-4}$ will be a conformally balanced HKT manifold, $H^B$ will satisfy (\ref{closx})  and $c_1^2$ must be a trivial class, where $c_1$ is the first Chern class of the principal bundle.

If the Lie algebra of $K$ is $\mathfrak{u}(1)\oplus \mathfrak{su}(2)$ and the representation $\mathrm{B}$ is non-trivial, the curvature of the principal bundle will decompose as
\be
  \mathcal{F}^{\mathfrak{u}(1)}=\mathcal{F}^{\mathfrak{u}(1)}_{\mathfrak{sp}(q-1)}~,~~~ \mathcal{F}^{\mathfrak{su}(2)}=\mathcal{F}^{\mathfrak{su}(2)}_{\mathfrak{sp}(q-1)}+\mathcal{F}^{\mathfrak{su}(2)}_{\mathfrak{sp}(1)}~,
  \label{decsp1}
  \ee
  with the component $\mathcal{F}^{\mathfrak{su}(2)}_{\mathfrak{sp}(1)}$ proportional to $\omega^\perp$ as a consequence of (\ref{frestrict}).  The base space $B^{4q-4}$ is a conformally balanced QKT manifold with $H^B$ restricted as in (\ref{closx}). The Pontryagin class of the principal bundle must be trivial.  Moreover, the adjoint bundle, $\mathrm{Ad}(M^{4q})$, of $M^{4q}$ is identified with the vector bundle over $B^{4q-4}$ spanned by the three complex structures that define the  QKT structure.
\end{theorem}

\begin{proof}

The proof of this statement has been given in \cite{gpheterotic}.  Here, we shall add a bit more explanation to establish some formulae that will be used later.
The focus will be on $K$ non-abelian and $\mathrm{B}$ a non-trivial representation of $K$. Without loss of generality, one can take the vector fields $Y$ to be orthonormal. As $Y_0=Y$ commutes with the rest. The Lie algebra brackets can be written as
\be
[Y, Y_r]=0~,~~~[Y_r, Y_s]=- h\, \epsilon_{rs}{}^t Y_t
\ee
where $h$ is a constant.  The only not trivial 3-dimensional representation of $\mathfrak{u}(1)\oplus \mathfrak{su}(2)$ is the 3-dimensional vector representation of $\mathfrak{su}(2)$. Thus, the representation $\mathrm{B}$ can be chosen as
\be
\mathrm{B}_0=0~,~~~(\mathrm{B}_r)^s{}_t=b\, \epsilon_r{}^s{}_t~,
\ee
where $b$ is a constant. Then, the requirement that $\mathrm{B}$ is a representation of the Lie algebra, $[\mathrm{B}_\alpha, \mathrm{B}_\beta]=-\tilde H^K_{\alpha\beta}{}^\gamma \mathrm{B}_\gamma$,
implies that $h=b$.

As $Y$ is tri-holomorphic, $\mathcal{F}^0=\mathcal{F}^{\mathfrak{u}(1)}$ is $(1,1)$-form with respect to all complex structures and solves (\ref{frestrict}),  which justifies the identification in (\ref{decsp1}). Moreover, the component of $\mathcal{F}$ along $\mathfrak{su}(2)$, $\mathcal{F}^{\mathfrak{su}(2)}$, decomposes as in (\ref{decsp1}). Then (\ref{frestrict}) implies that
\be
\mathcal{F}_{\mathfrak{sp}(1)}^r=-\frac{h}{2} (\omega^\perp)^r~.
\ee
The base space $B^{4q-4}$ is conformally balanced because $\pi^*\theta^B=\theta^\perp=-2d\Phi$ -- it can be seen by a direct calculations using the 3-form $H$ and the complex structures. The Pontryagin class of $M^{4q}$, as a principal bundle over $B^{4q-4}$, is trivial because of (\ref{closx}) as $H^B$ is a 3-form on $B^{4q-4}$.

\end{proof}

\begin{remark}
The conditions described in the theorem above is not the full set of conditions required for $M^{4q}$ to be an HKT manifold.  For this,  additional conditions are needed, in particular (\ref{rhorf}).  This now reads
 \be
 (\h\rho^B)^r=-h\, {\mathcal F}^r~.
 \label{rhorf2}
 \ee
 This relates the curvature of the base space to the curvature of the principal bundle connection.
\end{remark}

\begin{remark}
 It is evident from the above analysis that the associated adjoint bundle, $\mathrm{Ad}(M^{4q})$,  of $M^{4q}$, as a principal bundle over $B^{4q-4}$, is the vector subbundle of $\Lambda^2(B^{4q-4})$ spanned by the three
Hermitian structures $\omega^B_r$ of $B^{4q-4}$ viewed as a QKT manifold.

After taking a local section from $B^{4q-4}$ into $M^{4q}$, the QKT condition on $B^{4q-4}$ can be expressed as
\be
\h\nabla^B \omega^B_r-h \xi^t \epsilon_t{}^s{}_r \omega^B_s=0~,
\label{qktb}
\ee
where $\xi$ is the pull-back of the principal bundle connection $\lambda$ on $B^{4q-4}$ with the local section.
\end{remark}

  \subsection{An example} \label{sub:example}


The construction of HKT manifolds   as a fibration over a QKT manifolds has been known for sometime. The first such constructions have been in the context of homogeneous spaces in \cite{op}. Nevertheless, it is instructive to have an explicit example of a strong compact HKT manifold constructed in this way. One  simple example that illustrates the method is the HKT manifold $SU(3)$ \cite{Spindel}.  As a reward, the HKT structure on $SU(3)$ is explicitly presented.

First, let us construct the HKT structure on $SU(3)$ as a group manifold instead of as a fibration. Suppose that $\alpha$ and $\beta$ are the simple roots of $\mathfrak{su}(3)$. The positive roots  are $\Delta^+=\{\alpha, \beta, \alpha+\beta\}$ with the last one the highest, while the negative ones are $\Delta^-=\{-\alpha, -\beta, -\alpha-\beta\}$. Then,  the generators of $\mathfrak{su}(3)$ written in the Chevalley basis are $\{ h_p, e_\alpha, e_\beta, e_{\alpha+\beta}, e_{-\alpha}, e_{-\beta}, e_{-\alpha-\beta}; p=1,2\}$ and the non-vanishing components of the Killing form $\langle \cdot, \cdot \rangle$ can be chosen as
\be
\langle h_p, h_q\rangle=\delta_{pq}~,~~~\langle e_\gamma, e_{-\gamma'}\rangle=\delta_{\gamma\gamma'}~,
\label{kform}
\ee
where $\gamma, \gamma'\in \Delta^+$ and $\{h_p; p=1,2\}$ are the generators of the Cartan subalgebra. Note that $\alpha, \beta$ and $\gamma$ denote the roots of the Lie algebra $\mathfrak{su}(3)$ in contrast to their previous use as indices.

The first complex structure on $SU(3)$ acts on the above basis of the Lie algebra as
\be
I(h_1)=-h_2~,~~I(h_2)=h_1~,~~I(e_\gamma)=i e_\gamma~,~~~I(e_{-\gamma})=-i e_{-\gamma}~; ~~\gamma\in \Delta^+~,
\ee
i.e. the holomorphic basis is $\{h_1+i h_2, e_\alpha, e_\beta, e_{\alpha+\beta}\}$. It can be seen by an explicit calculation that $I$ is integrable -- the torsion $H$ constructed from the structure constants of $SU(3)$ and the Killing form is a $(2,1)\oplus (1,2)$ form.

The second complex structure is given by
\be
J(e_\alpha)=-e_{-\beta}~,~~~J(e_\beta)=e_{-\alpha}~,~~J(e_{\alpha+\beta})=h_1-ih_2~,~~J(e_{-\alpha-\beta})=h_1+i h_2~,
\ee
where the action of $J$ on the rest of generators is determined by the property that   $J^2=-{\bf 1}$.  It can be shown that $J$ is integrable and
anti-commutes with $I$. The Killing form together with $I$ and $J$ induce a left-invariant HKT structure on $SU(3)$.

Next, let us view the HKT structure on $SU(3)$ as arising from a fibration. It is well known that
\be
S(U(1)\times U(2))\hookrightarrow SU(3)\rightarrow \bC P^2~,
\label{fibr}
\ee
 with $S(U(1)\times U(2))=U(2)$
 as a group manifold. To describe  the inclusion $S(U(1)\times U(2))\hookrightarrow SU(3)$, observe that the group homomorphism  given by
 \begin{align}
 U(1)\times SU(2)&\rightarrow SU(3)
 \cr
 (z, A)&\rightarrow \begin{pmatrix}z^{-2} &0\cr
 0& z A\end{pmatrix}~,
 \end{align}
 where $z\in U(1)$ and $A\in SU(2)$, factors through $S(U(1)\times U(2))$ -- note that it has kernel $\bZ_2$.

  There are three different embedding of $SU(2)$ in $SU(3)$ up to a conjugation. These are given by the generators $\mathfrak{su}(2)_\gamma=\{h_\gamma, e_\gamma, e_{-\gamma}\}$ for $\gamma\in \Delta^+$, where $h_\gamma= \frac{2}{\gamma^2}\gamma^p h_p=\gamma^p h_p$ as the length square of all positive roots is $2$, $\gamma^2=2$. For every such embedding, the Lie algebra $\mathfrak{su}(3)$ decomposes as
\be
\mathfrak{su}(3)=\mathfrak{h}\oplus \mathfrak{m}~,
\ee
where $\mathfrak{h}=\mathfrak{su}(2)_\gamma\cup \bR\langle h_1, h_2\rangle$ and $\mathfrak{m}$ is spanned by the rest of the generators of $\mathfrak{su}(3)$. Moreover, for each such embedding, there is a fibration (\ref{fibr}). However, not all such fibrations are compatible with the HKT structure defined above on $SU(3)$. In particular,  notice that for $\gamma$ a simple root, the complex structure $J$ on $SU(3)$ does not respect the decomposition in fibre and base space directions.  It is only the embedding with $\gamma=\alpha+\beta$, the highest root that does.

Choosing the embedding labeled by the highest root, let us proceed to find the curvature $\mathcal{F}$ of the fibration.  These are given by the structure constants that determine the restriction of the commutator $[\mathfrak{m}, \mathfrak{m}]$ on $\mathfrak{h}=\mathfrak{su}(2)_{\alpha+\beta}\cup \bR\langle h_1, h_2\rangle$.  In particular, the relevant commutators are
\be
[e_\alpha, e_{-\alpha}]=h_\alpha~,~~[e_\beta, e_{-\beta}]=h_\beta~,~~[e_\alpha, e_\beta]=e_{\alpha+\beta}~,~~[e_{-\alpha}, e_{-\beta}]=-e_{-\alpha-\beta}~.
\ee
From this, one can read  the curvature $\mathcal{F}$  of the fibration as
\begin{align}
\mathcal {F}= \xi^\alpha\wedge \xi^{-\alpha} h_\alpha+\xi^\beta\wedge \xi^{-\beta} h_\beta+\xi^\alpha\wedge \xi^{\beta} e_{\alpha+\beta}-\xi^{-\alpha}\wedge \xi^{-\beta} e_{-\alpha-\beta}~,
\label{fcomp1}
\end{align}
where $\xi$ is the co-frame on $\bC P^2$ induced from left-invariant co-frame of $SU(3)$ after pulling back with a local section.  Note that $h_{\alpha-\beta}$ commutes with all those of $\mathfrak{su}(2)_{\alpha+\beta}$ and $(h_{\alpha-\beta}, h_{\alpha-\beta})=6$.

It can be easily seen that $H^B$ vanishes as those components of the structure constants of $SU(3)$ vanish.
As the torsion 3-form on $SU(3)$ is closed, consistency with (\ref{closx}) requires that $\mathcal {F}\wedge \mathcal {F}$ must identically vanish. Indeed, this can be verified by a short computation using (\ref{fcomp1}), where in the computation of $\mathcal {F}\wedge \mathcal {F}$ the  Killing form (\ref{kform})  on $SU(3)$ has been used to do the calculation    over the gauge Lie algebra generators.

Having proven that $\mathcal {F}\wedge  {F}$ vanishes identically for the construction considered, it is nevertheless instructive to notice that $\mathcal{F}$ in (\ref{fcomp1}) can be rewritten as
\begin{align}
\mathcal {F}&= \frac{1}{2} \Big((\xi^\alpha\wedge \xi^{-\alpha}-\xi^\beta\wedge \xi^{-\beta} )\, h_{\alpha-\beta}+(\xi^\alpha\wedge \xi^{-\alpha}+\xi^\beta\wedge \xi^{-\beta} )\, h_{\alpha+\beta}
\cr
&\qquad
+(\xi^\alpha\wedge \xi^{\beta}+\xi^{-\alpha}\wedge \xi^{-\beta})\,(e_{\alpha+\beta}-e_{-\alpha-\beta})
\cr
&
\qquad +(\xi^\alpha\wedge \xi^{\beta}-\xi^{-\alpha}\wedge \xi^{-\beta})\,(e_{\alpha+\beta}+e_{-\alpha-\beta})\Big)~.
\label{fcomp}
\end{align}
In this basis, $h_{\alpha-\beta}$ commutes with all the other basis elements of $\mathfrak{u}(1)\oplus\mathfrak{su}(2)$ and spans the $\mathfrak{u}(1)$ subalgebra while the rest of the basis elements span the $\mathfrak{su}(2)$ subalgebra. Moreover,  $h_{\alpha-\beta}$ has length square $6$ while the rest have $2$. It is expected from  Chern-Weyl  theory that in an expression of $\mathcal{F}\wedge \mathcal{F}$ in terms of characteristic classes  this relative factor of $3$ will make a contribution. Indeed, we shall demonstrate below that this is the case.

\subsection{Compact strong 8-dimensional HKT   manifolds}

\subsection{A refinement of theorems  \ref{th:pfibre} and \ref{th:pfibre2} }

Recently, it has been demonstrated   in \cite{bfgv} that for 8-dimensional, compact, simply connected,   strong HKT, non-hyper-K\"ahler,  manifolds $(M^8, g, H, I_r)$ the vector fields  $\{Y, Y_r=-I_r Y: r=1,2,3\}$ in Theorem \ref{th:pfibre2} generate a $\mathfrak{u}(1)\oplus \mathfrak{su}(2)$ Lie algebra and  preserve the span of the three complex structures of $M^8$ but they may not have closed orbits. Moreover, it has been shown that there is another $\mathfrak{u}(1)\oplus \mathfrak{su}(2)$ action on $M^8$ that has closed orbits but it may not preserve the span of the three complex structures.

In any case, the closure of $\{Y, Y_r=-I_r Y: r=1,2,3\}$ vector fields under Lie brackets for 8-dimensional HKT manifolds is a rather weak assumption.  If they do not close, $M^8$ will admit another four $\h\nabla$-covariantly constant vector fields and it will locally be isometric to a group manifold. Similarly, if the span of the three complex structures is not preserved under the action of $\{Y, Y_r=-I_r Y: r=1,2,3\}$, there will be additional $\h\nabla$-covariantly constant $(1,1)$ tensors on $M^8$ -- the Lie derivative of the complex structures -- that will reduce the holonomy of $\h\nabla$ to identity. This again leads  to a local group manifold structure on $M^8$.  Both these results are a consequence of the first Bianchi identity in (\ref{bianchi}), see also discussion in \cite{uggp, phgp3, gpheterotic}.

Therefore, the only essential assumption that we have to make to apply the results of the Theorems \ref{th:pfibre} and \ref{th:pfibre2} is that the action generated by the  fields $\{Y, Y_r=-I_r Y: r=1,2,3\}$ on $M^8$ can be integrated to a free action of a compact group $K$ on $M^8$. As a result  $M^8$  is a principal bundle with fibre group $K$.

The results of the Theorems  \ref{th:pfibre} and \ref{th:pfibre2} can be simplified for 8-dimensional HKT manifolds. Focusing on the {\it non-abelian group} $K$, the base space $B^4$ admits a conformally balanced QKT structure. It turns out that in such a case, the metric on $B^4$ can be re-scaled  as
\be
g^B= e^{2\Phi} \mathring {g}~,
\label{gbmg}
\ee
such that $\mathring {g}$ is a quaternionic K\"ahler metric. In particular, the condition (\ref{qktb}) now written in terms of the quaternionic structure reads
\be
\mathring \nabla I^B_r-h\, \xi^t \epsilon_t{}^s{}_r I^B_s=0~.
\label{qkxx}
\ee
where $\mathring \nabla$ is the Levi-Civita connection of $\mathring g$ on $B^4$.

Moreover, after a short calculation, the closure of $H$ given in   (\ref{closx})   can be rewritten  as
\begin{align}
\mathring{\nabla}^2 e^{2\Phi}&
=-\frac{1}{2} \Big(\mathcal{F}^{\mathfrak{u}(1)})_o^2+ (\mathcal{F}_-^{\mathfrak{su}(2)})_o^2\Big)+ \frac{3 h^2 }{2}  e^{4\Phi}~,
\label{dconx}
\end{align}
where $(\cdot )_0^2$ denotes the square of the tensor is  taken with respect to the metric $\mathring{g}$. This equation has appeared before in a similar context in \cite{bgp}. In the last line, we have also decomposed
$\mathcal{F}^{\mathfrak{su}(2)}$ into anti-self-dual and self-dual components as
\be
\mathcal{F}^{\mathfrak{su}(2)}=\mathcal{F}^{\mathfrak{su}(2)}_-\oplus \mathcal{F}^{\mathfrak{su}(2)}_+~,
\ee
with the self-dual component
\be
\mathcal{F}^{\mathfrak{su}(2)}_+=-\frac{h}{2} \omega^r~.
\label{fomega}
\ee
The anti-self-dual component $\mathcal{F}^{\mathfrak{su}(2)}_-$ is a $(1,1)$-form with respect to all complex structures and $\mathcal{F}^{\mathfrak{u}(1)}=\mathcal{F}^{\mathfrak{u}(1)}_-$, i.e. we have chosen the orientation on $B^4$ induced by the quaternionic K\"ahler structure with $d\mathrm{vol}=\frac{1}{3} \sum_r \omega^B_r\wedge \omega^B_r$.
We have changed notation from that in (\ref{decsp1}) because in four dimensions, $q=2$, it is not possible to distinguish between $Sp(q-1)$ and $Sp(1)$.  In appendix B, we shall show that (\ref{dconx}) has always a solution on  $B^4$, compact,  with fixed metric $\mathring g$ and $\mathcal{F}$ using the method of supersolutions and subsolutions.

\begin{lemma}\label{le:positive}
Suppose that $(M^8, g, H, I_r)$ is a compact, strong, HKT manifold that satisfies the conditions of Theorem \ref{th:pfibre2}. Then, $(B^4, \mathring g, I_r)$ is a anti-self-dual 4-manifold with positive scalar curvature.
\end{lemma}
\begin{proof}
A 4-manifold is anti-self-dual, iff the self-dual part of the Weyl tensor vanishes. From Proposition 7.1, page 92, in \cite{salamon}, it is sufficient to show that
$\mathcal{F}^{\mathfrak{su}(2)}_+(Q)$ is given in terms of the Ricci scalar $R(\mathring g)$ as detailed below.   From (\ref{fomega}), one has
 $\mathcal{F}^r_+=-\frac{h}{2}\omega^r$, which can be rewritten as  $\mathcal{F}^r(Q)_+=-\frac{h}{2} e^{2\Phi} \mathring{\omega}^r$.

As the HKT manifold $M^8$ is a gradient generalised Ricci soliton, it satisfies the condition $\h \R+2\h\nabla\partial \Phi=0$.  Thus  $\h R+2\nabla^2\Phi=0$. Next, we express the  scalar curvature  $\h R$ of the 8-dimensional manifold $M^8$ in terms of the Ricci scalar $\h R(g^B)$ of the base space $B^4$ using the identity (4.25) in \cite{gpheterotic}.  This gives
\be
2(\nabla^B)^2 \Phi+ R(g^B)-\sum_{\alpha=0}^3 ({\mathcal F}^r)^2=0\Longrightarrow 2(\nabla^B)^2 \Phi+ R(g^B)-({\mathcal F}^{\mathfrak{u}(1)})^2-
({\mathcal F}^{\mathfrak{su}(2)}_-)^2-3 h^2=0~.
\ee
 Next, we substitute for $g^B=e^{2\Phi} \mathring g$ and after some computation, the above expression can be rewritten as
 \be
 \mathring{\nabla}^2 e^{2\Phi}=\frac{1}{2} R(\mathring g)-\frac{1}{2} \Big(({\mathcal F}^{\mathfrak{u}(1)})_o^2+ ({\mathcal F}^{\mathfrak{su}(2)}_-)_o^2\Big)
 -\frac{3 h^2}{2} e^{4\Phi}~,
 \label{feqn}
 \ee
where we have  used the standard relation
\be
R(e^{2\Phi} \mathring g)=e^{-2\Phi} \Big(R(\mathring g)- 6\mathring \nabla^2 \Phi-6 (\partial\Phi)^2_o\Big)~,
\ee
 in four dimensions.

 Upon comparing,  (\ref{dconx}) with (\ref{feqn}), we  conclude that
 \be
 R(\mathring g)=6 h^2 e^{2\Phi}>0~.
 \label{rdil}
 \ee
 Thus, the  scalar curvature of $(B^4, \mathring g)$ must be strictly positive.  Note also that this equation relates the radius of the fibre with that of the base space $B^4$ for fixed $\Phi$.

 It is well known that the Weyl tensor $W$ and  $\mathcal{F}_+^r$  determine  the self-dual component of the Riemann tensor that maps $\Lambda^{2+}(B^4)$ into $\Lambda^{2+}(B^4)$.  In particular, we find that  $h\mathcal{F}_+^r=-\frac{1}{12 } R(\mathring g) \mathring\omega^r$, which is the required condition for $B^4$ to be anti-self dual.  Note that the curvature of $\Lambda^+(B^4)$ is $h \mathcal{F}_+^r$ from (\ref{qkxx}).
\end{proof}

Next, we  consider  principal bundles with fibre group $K$ either $S(U(1)\times U(2))$ or $U(2)$. Following the $SU(3)$ example, we shall investigate the former case. The analysis for the latter is similar as it will be explained below.
Principal bundles with fibre group $K=S(U(1)\times U(2))$ over a 4-manifold $B^4$ are associated with a complex vector bundle $E$ and a complex line bundle $L$ that satisfy the  relation
\be
c_1(E)+c_1(L)=0~,
\label{trcon}
\ee
where $c_1(E)$ and $c_1(L)$ are the first Chern classes of $E$ and $L$, respectively.  This imposes the condition that the fibre group $K$ is special. As a result, they are classified by  $(c_1(L), c_2(E))\in H^2(B^4, \bZ)\oplus H^4(B^4, \bZ)$, where $c_2(E)$ is the second Chern class of $E$.

\begin{lemma}\label{le:top}
The topological condition  $[\mathcal{F}\wedge \mathcal{F}]=0$ required for the existence a strong  HKT structure on the principal bundle $M^8$ with fibre group $S(U(1)\times U(2))$ over a manifold $B^4$ can be expressed in terms of the
triviality of the class
\be
3 c_1(L)^2+ 2\chi+3\tau~,
\ee
where $\chi$ and $\tau$ are the Euler and signature characteristic classes of $B^4$, respectively. Thus $(2\chi+3\tau)[B^4]$ must be divisible  by $3$.  It is also required that $c_1(L)^2[B^4]\leq 0$ in the orientation defined by the quaternionic structure on $B^4$.
\end{lemma}
\begin{proof}

We have demonstrated that in order a principal bundle $M^8$ with fibre group $S(U(1)\times SU(2))$ to admit an HKT structure, the associated  adjoint bundle, $\mathrm{Ad}(M^8)$, of the principal bundle $M^8$ must be identified with $\Lambda^+(B^4)$, see (\ref{qkxx}).
This implies that the adjoint bundle of $E$, $\mathrm{Ad}(E)\equiv \mathrm{Ad}(M^8)$,  should be identified with $\Lambda^+\equiv \Lambda^+(B^4)$ as the adjoint representation of abelian groups is trivial.  Thus from the Hirzenbruch's signature theorem
\be
p_1(\mathrm{Ad}(E))=p_1(\Lambda^+)=2\chi+3\tau~,
\label{pone}
\ee
where $\chi$ is the Euler characteristic class and $\tau$ is that of the signature class of $B^4$.

A classic formula in characteristic classes expresses the Pontryagin class of $\mathrm{Ad}(E)$ in terms of the Chern classes of $E$ as
\be
p_1(\mathrm{Ad}(E))=c_1(E)^2-4 c_2(E)~.
\ee
Thus, we conclude that
\be
c_1(E)^2-4 c_2(E)=2\chi+3\tau~.
\label{ccct}
\ee
Using the splitting principle for the $E$ bundle, we have that $c_1(E)=x_1+x_2$,  $c_2(E)=x_1x_2$ and $p_1(E)=x_1^2+x_2^2$. Set also $c_1(L)=\ell$. In such a case, we have
 \be
 \ell=-x_1-x_2~,
 \ee
 from (\ref{trcon}).
 Furthermore, after an appropriate overall normalisation, typically $1/4\pi^2$ for principal bundle connections, the class
\be
[\mathcal{F}\wedge \mathcal{F}]=\ell^2+x_1^2+x_2^2=2\ell^2-2x_1x_2=2\ell^2+\frac{1}{2} \Big(2\chi+3\tau-\ell^2\Big)=\frac{1}{2} \Big(3\,c_1(L)^2+2\chi+3\tau\Big)~,
\ee
where we have used (\ref{ccct}).  This proves the first part of the lemma after possibly an overall rescaling of the formula. The condition that $c_1(L)^2[B^4]\leq 0$ follows from the fact that the associated curvature is anti-self-dual in the orientation on $B^4$ induced by the quaternionic structure.

\end{proof}

\begin{remark}
Specifying the base space manifold $B^4$, the principal bundles with $S(U(1)\times U(2))$ fibre that admit an HKT structure are only characterised by $c_1(L)\in H^2(B^4, \bZ)$.  This is because $c_2(E)$ is determined in terms of $c_1(L)$, $\chi$ and $\tau$ as a result of equations (\ref{trcon}) and (\ref{ccct}).  Therefore in the search of examples with a given  base space $B^4$,  the only choice left is to vary $c_1(L)$ for the search of topologically different manifolds $M^8$.  Furthermore, the topological condition arising from the closure of $H$ determines $c_1(L)^2[B^4]$ in terms of the topological data of the base space.  This considerably restricts the possibility of a solution.

\end{remark}

\begin{remark}
 The condition (\ref{rhorf2}) that relates the Ricci forms of the base space $B$ to the curvature $\mathcal F$ of the principal bundle can now be rewritten
 as
 \be
 \mathring \rho^r=-h\, \mathcal{F}^r~,
 \label{rhorf3}
 \ee
 where $\mathring \rho^r$ are the Ricci forms of the Quaternionic K\"ahler manifold $(B^4, \mathring g)$.  This can be used, together with (\ref{rdil}), to rewrite
 either (\ref{dconx}) or (\ref{feqn}) in terms of the geometry of $(B^4, \mathring g)$ after eliminating some of the contribution of $\mathcal {F}$.  We shall demonstrate below that the resulting differential  equation  will restrict the geometry of base space.

\end{remark}

\begin{remark}

In the search for manifolds with an HKT structure, one can also consider  principal bundles $M^8$ with fibre group $U(2)$. In such a case, a similar analysis to that presented in Lemma \ref{le:top} will reveal that, up to an appropriate overall normalisation of the first term,
\be
 [\mathcal{F}\wedge \mathcal{F}]=p_1(E)=3c_1(E)^2+2\chi+3\tau~,
\ee
where $E$ is the associated rank 2 complex bundle.  Again, given the base space $B^4$, the only choice left to determine the topology of the HKT manifold $M^8$ as a principal bundle  is $c_1(E)\in H^2(B^4, \bZ)$.

\end{remark}

\begin{prop}\label{prop:HKT}
 A principal bundle $M^8$ with fibre group either $S(U(1)\times U(2))$ or $U(2)$ such that the action of the fibre group preserves the span of the three complex structures admits a strong HKT structure, iff the following conditions hold:
 \begin{enumerate}
 \item The base space $(B^4, \mathring g)$ of the principal bundle  is an anti-self-dual 4-manifold with positive scalar curvature.

  \item  The topological condition $(3c_1^2+2\chi+3\tau)[B^4]=0$ holds, where  $c_1$ is the first Chern class  that characterises bundles with fibre  $S(U(1)\times U(2))$  or the first class of the $U(2)$ bundle, respectively,  and $\chi$ and $\tau$ are the E\"uler and signature characteristic classes of $B^4$.

 \item The Chern class $c_1$ satisfies the condition $c_1^2[B^4]\leq 0$  in the orientation given by the quaternionic structure.

 \item The condition $\mathring \rho^r=-h\, \mathcal{F}^r$ is valid.

 \item The metric $\mathring g$ on $B^4$ solves the equation
 \begin{equation}
 \mathring{\nabla}^2 R(\mathring g)=-3 h^2 \big(\mathcal{F}^{\mathfrak{u}(1)}\big)_o^2-3 \big(\mathrm{Ric}(\mathring g)\big)^2_o + R(\mathring g)^2~.
 \label{rog}
 \end{equation}
 \end{enumerate}
 \end{prop}

\begin{proof}
To prove this in one direction observe that the statements 1, 2 and 3 of the proposition follow from the two Lemmas \ref{le:positive} and \ref{le:top}. We have also explained how the fourth condition arises.

The last statement follows from (\ref{dconx}). For this, first express   $e^{2\Phi}$ in terms of the Ricci scalar using (\ref{rdil}).  Then, the integrability condition of (\ref{qktb}) gives, after some computation, that
\be
\mathrm{Ric}(\mathring g)(V, W)= -h \sum_{r=1}^3 \mathcal{F}^r(I_r^B V, W)~.
\ee
This in turn implies that
\be
\Big(\mathrm{Ric}(\mathring g)-\frac{1}{4} \mathring g\, R(\mathring g)\Big)_o^2=h^2 \big(\mathcal{F}_-^{\mathfrak{su}(2)})\big)_o~,
\ee
i.e. the length square of the anti-self-dual part of the $\mathfrak{su}(2)$ component of the curvature of the principal bundle can be expressed in terms of the length square of the traceless part of the Ricci curvature of the metric  $\mathring g$ of the base space $B^4$. Putting this into (\ref{dconx}) and after a re-arrangement of terms, one arrives at (\ref{rog}).

The converse follows because an HKT structure on the bundle space $M^8$ can be constructed using (\ref{hktdec}) together with (\ref{gbmg}). The restriction for the base space $B^4$ to be anti-self-dual leads to the relation (\ref{rdil}).  Furthermore, the fourth condition implies that the Ricci forms of $M^8$ vanish as it is required for an HKT manifold. The only condition that remains to be satisfied is the closure of $H$.  This is a consequence of the statement 5 in the proposition.
\end{proof}

It is clear that the existence of compact 8-dimensional strong HKT manifolds depends on the existence of compact anti-self-dual 4-dimensional manifolds with positive scalar curvature. It turns out that these restrictions  impose  rather strong conditions on  $B^4$. Adapting these results to the notation and the orientation of $B^4$, we use,  the statement in \cite{lebrun0, lebrun} reads as follows:

\begin{prop}\label{prop:lb}
Let $(B^4, \mathring g)$  be a compact simply connected anti-self-dual $4$-manifold with positive scalar curvature. Then, $B^4$ is homeomorphic either to $\#_k\overline{\bC P}^2$ for some $k\geq 1$ or to $S^4$. Moreover, $B^4$  is diffeomorphic to $\#_k\overline{\bC P}^2$ for $k\leq 4$.
\end{prop}

\begin{remark}

As $B^4$ is simply connected the first Betti number $b_1$ vanishes -- in fact $b_1=0$ is a sufficient condition for the proof of the Proposition \ref{prop:lb} above and of the Theorem \ref{th:cla} below provided we consider the universal cover of $B^4$ instead of $B^4$. Positivity of the scalar curvature and  anti-self-duality  of $B^4$  imply that $b^2_+=0$ -- this relies on a Weitzenb\"ock formula to rule out harmonic self-dual 2-forms on $B^4$.   Then, the above theorem is a consequence of the results of Donaldson and Freedman.

\end{remark}

{\bf Proof\, of \, Theorem\, \ref{th:cla}}.

Let us first consider the case that either $K$ is abelian or $K$ has Lie algebra $\mathfrak{u}(1)\oplus \mathfrak{su}(2)$ and acts with the trivial representation $\mathrm{B}$ on the three complex structures on $M^8$. In either case, $\mathcal{F}_+=0$ and so $\mathcal{F}=\mathcal{F}_-$. Then  the closure condition of $H$, see (\ref{dconx}),  implies that
\begin{align}
\mathring{\nabla}^2 e^{2\Phi}&=-\frac{1}{2} (\mathcal{F})_o^2~.
\label{dconx3}
\end{align}
Thus, after either integrating the above expression  over $B^4$ or using the Hopf maximum principle, we conclude that $\Phi$ is constant and $\mathcal{F}=0$. The principal bundle is trivial and so $M^8$ is  isometric to the product $B^4\times K$, with  both $B^4$ and $K$ either  strong HKT manifolds or hyper-K\"ahler. As the compact 4-dimensional hyper-K\"ahler manifolds are locally isometric and tri-holomorphic to $\bR^4$ or to $K_3$ and the only compact, non-hyper-K\"ahler, 4-dimensional  HKT manifold is locally isometric and tri-holomorphic to $\bR\times S^3$, the result follows.

Next, suppose that $K=S(U(1)\times U(2))$ or $K=U(2)$  and the representation $\mathrm{B}$ is non-trivial.  This is the case dealt in Proposition \ref{prop:HKT}.
As $B^4$ is anti-self-dual manifold with positive scalar curvature, the Proposition \ref{prop:lb} implies that $B^4$ is homeomorphic to either $S^4$ or to  $B^4=\#_k\overline{\bC P}^2$. It is easy to see that $S^4$ does not satisfy the topological condition of the Theorem \ref{prop:HKT}. There are no non-trivial circle bundles over $S^4$, thus    $c^2_1[S^4]=0$, and moreover $\chi[S^4]=2$ and $\tau[S^4]=0$. For $B^4=\#_k\overline{\bC P}^2$, the topological condition of the Theorem \ref{th:HKT} reads
\be
3\sum^k_{p=1} n_p^2= 4-k~,
\label{inteqn}
\ee
where $c_1(L)$ has been expanded in a basis of $H^2(B^4, \bZ)$ with components $\{n_p\in \bZ: p=1,\dots, k\}$. We have also used that the intersection  form of $\#_k\overline{\bC P}^2$ is $-\bf{1}$.

The equation (\ref{inteqn})  has solutions for $k=1$ that gives $B^4=\overline{\bC P}^2$ and  $k=4$ that gives $B^4=\#_4\overline{\bC P}^2$. For the former solution,  $M^8$ is diffeomorphic to $SU(3)$, $M^8=SU(3)$. An HKT structure on $SU(3)$ has been described   in section \ref{sub:example}.

In the latter case, $c_1(L)[B^4]=0$ and as a consequence of (\ref{trcon}) and (\ref{ccct}), $c_2(E)=0$.  Thus, $M^8$ is a topologically trivial bundle over $B^4$. But, $M^8$ cannot be an HKT manifold. This is because it will be homeomorphic to a product $\#_4 \overline{\bC P}^2\times S^1\times S^3$. As a result, it will not admit a spin structure as $\#_4 \overline{\bC P}^2$ is not spin manifold. However, all HKT manifolds admit a spin structure.

\hfill $\square$

\begin{remark}
We have demonstrated in Theorem \ref{th:cla} that one of the possibilities that arise for an 8-dimensional,  compact, strong, HKT manifold is to be diffeomorphic to $SU(3)$.
As $SU(3)$ is a compact, simply connected, group manifold, it follows from the Borel–Hsiang–Shaneson–Wall theorem that it admits a unique differentiable structure -- it is that constructed using the group multiplication law. We also know that $SU(3)$ admits a left-invariant HKT structure that described in detail in section \ref{sub:example}. The question is whether there are additional HKT structures on $SU(3)$, especially HKT structures that are not left- or right-invariant.
It is a consequence of theorem 2.7 in \cite{poon} that if there is another HKT structure on $SU(3)$ it must be in the conformal class of the Fubini-Study metric.
\end{remark}

\begin{corollary}\label{coro:one}
If $SU(3)$ admits another HKT structure in addition to the invariant ones, the additional HKT structure must be induced by a anti-self-dual structure on $\overline{\bC P}^2$ that lies in the conformal class of the Fubini-Study metric.
\end{corollary}
\begin{proof}
This is a consequence of theorem 2.7 of Poon in \cite{poon}. In particular, it has been demonstrated that all compact simply connected anti-self-dual manifolds with signature $-1$ and positive scalar curvature are conformally equivalent to $\overline{\bC P}^2$ with the Fubini-Study metric.  As, we have shown that all HKT structures that obey the conditions of theorem \ref{th:cla} are induced by  anti-self-dual structures on $\overline{\bC P}^2$ with positive scalar curvature, it follows that they must be constructed  from  anti-self-dual structures conformal to the Fubini-Study metric.
\end{proof}

\begin{remark}
If $SU(3)$ is equipped with a left invariant HKT structure, it is possible to construct many other HKT manifolds by choosing a discrete subgroup $D$ of $SU(3)$ and considering the manifold $M^8=D\backslash SU(3)$.  As the metric $g$,  torsion $H$ and complex structures $I_r$  of $SU(3)$ are  left- invariant, they  will induce an HKT structure on $D\backslash SU(3)$.  It is also possible to consider $M^8=SU(3)/D$. The metric and torsion of $SU(3)$ will descent to $SU(3)/D$ as they are bi-invariant.  However, the complex structures $I_r$ must be invariant under $D$, i.e. the action of the discrete group must be tri-holomorphic.
\end{remark}

\begin{remark}
The results of the above Theorem \ref{th:cla} can  be  extended to include the case that the action of $\mathfrak{u}(1)\oplus \mathfrak{su}(2)$ can be lifted to a free action of $S(U(1)\times U(2))/\Gamma$ on $M^8$, where $\Gamma$ is a finite normal subgroup. In fact, $\Gamma=\bZ_k$ and  coset group is either $S(U(1)\times U(2))$ ($k$ odd) or $U(1)\times SO(3)$ ($k$ even). So essentially, there is only one more case to consider that of $U(1)\times SO(3)$ -- note that $U(1)\times SO(3)$ also admits invariant HKT structures associated with a bi-invariant metric and 3-form. This case has been explored in \cite{gphh}, where further details can be found.
\end{remark}

\vskip1cm
 \noindent {\it {Acknowledgements}}

 I thank Simon Salamon for discussions,  Leander Stecker for bringing to my attention reference \cite{AFF}  and Jeffrey Streets for discussions and for point out a mistake in an earlier version of the paper.

\setcounter{section}{0}

\appendix{Notation}\label{ap:aa}

Let $(M^n,g,H)$ be a Riemannian manifold with metric $g$ and 3-form $H$. The connection with torsion $H$, $\h\nabla$, expressed in an orthonormal co-frame $\{e^i: i=1,\dots, n\}$, is defined as
\be
\h\nabla_j X^i=\nabla_j X^i+\frac{1}{2} H^i_{jk} X^k~,
\ee
where $X$ is a vector field and  $\nabla$ is the Levi-Civita connection of $g$. In this co-frame, $g=\delta_{ij} e^i e^j$ and we use Einstein notation, where repeated indices are summed over. We also use the same symbol to denote the 3-form $H$ with all the other tensors that can be constructed from $H$ with the use of the metric $g$ and its inverse.  For example, $g^{-1} H$ with components  $H^i{}_{jk}= g^{im} H_{mjk}$. It will be apparent from the context which tensor $H$ denotes.

 On the other hand, if $X$ is a vector field then $X^\flat$ denotes the associated 1-form $X^\flat(Y)=g(Y,X)$. Similarly, if $\chi$ is an 1-form, the $\chi^\flat$ denotes the associated vector field, i.e. $\chi(Y)=g(Y, \chi^\flat)$.

In a coordinate basis $\{x^I: I=1,\dots, n\}$,  the convention for curvature $\h R$ of $\h\nabla$, as well as that  of other connections, is
\be
(\h\nabla_I \h\nabla_J-\h\nabla_J \h\nabla_I) X^K= \h R_{IJ}{}^K{}_L X^L- H^L{}_{IJ} \h\nabla_L X^K~.
\ee
In the co-frame basis,  $\h R_{ijkm}=\delta_{kn} \h R_{ij}{}^n{}_m$.  Furthermore, the Ricci tensor (curvature) is defined as
\be
\h{\mathrm{Ric}}_{ij}=\h R_{k i}{}^k{}_j~.
\ee
The wedge product of a $p$-form $\chi$ with a $q$-form $\psi$  is defined in the usual way.  In particular in the co-frame basis reads as
\begin{align}
\chi\wedge \psi&=\frac{1}{p!} \frac{1}{q!} \chi_{i_1\dots i_p} \psi_{j_1\dots j_q}\, e^{i_1}\wedge\cdots \wedge e^{i_p}\wedge e^{j_1}\wedge \cdots e^{j_q}
\cr
&
= \frac{1}{p!} \frac{1}{q!} \chi_{[i_1\dots i_p} \psi_{i_{p+1}\dots i_{p+q}]}\, e^{i_1}\wedge\cdots \wedge  e^{i_{p+q}}~.
\end{align}
In particular, for $p=1$ and $q=2$,
\be
3 \chi_{[i_1} \psi_{i_2i_3]}\equiv \chi_{i_1} \psi_{i_2i_3}+ \mathrm{cyclic} (i_1, i_2, i_3)~,
\ee
and similarly for $p=2$ and $q=1$.

The inner product of two $p$-form $\chi$ and $\psi$ is defined as
\be
(\chi, \psi)=\frac{1}{p!} \chi_{i_1\dots i_p} \psi_{j_1\dots j_p}\, \delta^{i_1 j_1} \cdots \delta^{i_p j_p}~.
\ee
We also use $(\chi, \chi)= \frac{1}{p!} \chi^2$ and $|\chi|=\sqrt{(\chi, \chi)}$.

The Lee form $\theta$ of a KT manifold $(M^{2k}, g, I, H)$ is given by
\be
\theta_i=\nabla^k \omega_{kj} I^j{}_i=-\frac{1}{2}  H_{kjm} \omega^{kj} I^m{}_i~,
\ee
where $\omega(X,Y)=g(X, I Y)$ is the Hermitian form.

\appendix{Existence of solutions to a non-linear elliptic equation}\label{ap:b}

It is intriguing that the equation (\ref{dconx}) always admits smooth  positive solutions for $e^{2\Phi}$.
After  re-arranging  this equation, with possibly a constant re-scaling of the metric $\mathring g$,  we find
\be
G(u)\equiv -\mathring{\nabla}^2 u(x)+u(x)^2-w(x)=0~,
\label{diffeqn}
\ee
where we have put the differential equation in an elliptic form,  $u(x)=e^{2\Phi}$ and $w(x)\geq 0$ with $w(x)\not=0$ for all $x\in B^4$.
In fact, the proof applies to all dimensions but we shall give it on a 4-dimensional manifold.

\begin{prop}\label{prop:dilaton}
If $B^4$ is compact  and $w$ smooth,  (\ref{diffeqn}) admits a smooth solution $u(x)$ with $u(x)>0$.
\end{prop}
\begin{proof}
 This equation can be solved using the approximation method based on supersolutions and subsolutions. For this, rewrite the differential equation $G(u)=0$ as
 \be
 L(u)=-u^2+w(x)+\lambda u~,~~~L\equiv -\nabla^2+\lambda
 \ee
 where $\lambda$ is a positive constant, $\lambda>0$, to be determined later.  Clearly if $u$ satisfies this equation, it also solves the original equation $G(u)=0$.
 As $B^4$ is compact, the function $w$ attaints a maximum and a minimum value $w_{\mathrm{max}}$ and $w_{\mathrm{min}}$, respectively. Let
 \be
 a=\frac{1}{ 2} \sqrt{w_{\mathrm{min}}}~,~~~b=\sqrt{w_{\mathrm{max}}}~.
 \ee
 These constants can be used  as a subsolution and as a supersolution, respectively, because
 \begin{align}
 &G(a)=a^2-w(x)\leq a^2-w_{\mathrm{min}}=-\frac{3}{4} w_{\mathrm{min}} \leq 0~,
 \cr
 &G(b)=b^2-w(x)\geq b^2-w_{\mathrm{max}}=0~,
 \end{align}
 and $a<b$. Thus, any possible solution resulting from the approximation will take values  in the interval $[a, b]$.

 Before, we begin the iteration, let us establish some properties of the operator $L$ and the function
 \be
 T(u)\equiv L^{-1} f(u,w)~,~~~f(u,w)=w-u^2+\lambda u~,
 \ee
 where $L^{-1}$ is the inverse of $L$. As $L$ can be inverted on any $C^{k,\alpha}$ H\"older space because $\lambda>0$ and the eigenvalues of $-\nabla^2$ are all non-positive, the function $T$ is well-defined on $C^{k,\alpha}$, $k\geq 2$. In particular, $T$ maps $C^0$ into $C^{2,\alpha}$. One also has that
 \begin{align}
  &L(T(a))=w-a^2+\lambda a\geq \lambda a= L(a)~,
  \cr
  & L(T(b))=w-b^2+\lambda b\leq \lambda b= L(b)~.
  \label{fite}
  \end{align}

 Next, let us choose $\lambda$ such that $f(u,w)$ is monotone increasing in the variable $u$ in the interval $[a,b]$.
  Taking the derivative with respect to $u$, one finds
 \be
 \partial_u f(u,w)=-2u+\lambda~.
 \ee
 Choosing $\lambda>2b$, $\partial_u f(u,w)>0$ and thus it is increasing.  Moreover for this choice of $\lambda$, $L$ remains invertible and satisfies the maximum principle as it is explained below.

 The operator $L$ is order preserving. This means that if $v_i$ are solutions to the equation $L(v_i)=h_i$ for the functions with $h_1\leq h_2$ pointwise on $B^4$, then $v_1\leq v_2$. This is a consequence of the strong maximum principle. Indeed, suppose that there are points that $v_1-v_2$ attains a positive value, i.e $v_1>v_2$ at those points. This implies that there is an $x_0\in B^4$ that $v_1-v_2$ attains a maximum value which is positive. Evaluating $L$ on $v_1-v_2$ at $x_0$, one has
 \be
 L(v_1-v_2)\vert_{x=x_0}\geq \lambda (v_1-v_2)(x_0)\geq 0~,
 \ee
 as $\nabla^2 (v_1-v_2)\vert_{x_0}\leq 0$.
 But from the hypothesis that $h_1\leq h_2$,  $L(v_1-v_2)=h_1-h_2\leq 0$. This is a contradiction and so $v_1\leq v_2$ at every point.

 As $L$ is order preserving, the same applies for $L^{-1}$.  As a result, the function $T(u)=L^{-1} f(u, w)$ is increasing as $u$ increases. This is because $f$ is also an increasing function of $u$.

 The map $T$ maps the interval $[a,b]$ to itself. Indeed, consider
 \be
 L(T(a)-a)=f(a, w)-L(a)=(w-a^2+\lambda a)-\lambda a=w-a^2\geq 0~.
 \ee
 From the monotonicity of $L$ proven above, $T(a)-a\geq 0$ and so $T(a)\geq a$.  Similarly,
 \be
 L(b-T(b))=L(b)-f(b,w)=\lambda b-(w-b^2+\lambda b)=-w+b^2\geq 0~.
 \ee
 Thus,  $b-T(b)\geq 0$ and so $T(b)\leq b$. Therefore, $T$ maps the closed convex set
 \be
 S\equiv \{u\in C^0(B^4): a\leq u\leq b\}
 \ee
 into itself and it is order-preserving on $S$.

 Now, we are ready to perform the iteration. This is  defined as
 \be
 L(u_{n+1})= f(u_n) \Longleftrightarrow u_{n+1}= T(u_n)~.
 \ee
For (\ref{fite}), it is clear that $u_1\geq u_0$ and progressing inductively using the properties of $T$, one can establish that   $a=u_0\leq u_1\leq u_2\cdots \leq b$. So the sequence $(u_n)$ is monotone nondecreasing and uniformly bounded in $C^0(B^4)$. As $T$ maps into $C^{2,\alpha}(B^4)$, the sequence actually lies in a bounded subset of $C^{2,\alpha}(B^4)$. By Arzel\'a-Ascoli and the compact embedding of $C^{2,\alpha}$ into $C^0$, $(u_n)$ converges to a limit $u\in C^0(B^4)$. Passing to the limit we obtain pointwise $L(u)=f u$. Elliptic regularity implies that actually $u$ is smooth provided that $w$ is smooth.

 As $0<a\leq u\leq b$, the solution $u$ must be strictly positive everywhere on $B^4$.  This means that there is a smooth function $\Phi$ such that $u(x)=e^{2\Phi(x)}$.

\end{proof}

 \vskip0.5cm
\noindent{\it Data availability statement}

All data generated or analyzed during this study are included in this published article.

\vskip0.5cm
\noindent{\it Conflict of interest statement}

The author has no relevant financial or non-financial interests to disclose.

\vskip0.5cm

 \bibliographystyle{unsrt}

\end{document}